\documentclass[12pt,english]{article}

\usepackage[T1]{fontenc}
\usepackage[utf8]{inputenc}
\DeclareUnicodeCharacter{202F}{\,}
\usepackage[main=english]{babel}

\usepackage{csquotes}
\usepackage[backend=biber,maxbibnames=99,maxcitenames=2]{biblatex}
\renewbibmacro{in:}{}
\DeclareFieldFormat[article]{volume}{\mkbibbold{#1}}
\renewbibmacro*{volume+number+eid}{%
  \printfield{volume}%
  \iffieldundef{number}
    {}
    {\addcomma\space \printtext{no.~\printfield{number}}}%
  \setunit{\addcomma\space}%
  \printfield{eid}%
}

\usepackage{geometry}
\usepackage{amsmath,amsthm,amssymb,amsfonts,mathtools}
\usepackage{mathrsfs}
\usepackage{esint}

\usepackage{graphicx}   
\usepackage{float}
\usepackage{tikz}       
\usepackage{pgfplots}
\usepackage{caption}
\usepackage{subcaption}

\usetikzlibrary{intersections, pgfplots.fillbetween}

\usepackage{comment}
\usepackage{indentfirst}
\usepackage{url}

\theoremstyle{definition}
\newtheorem{theorem}{Theorem}[section]
\newtheorem{lemma}[theorem]{Lemma}
\newtheorem{remark}[theorem]{Remark}
\newtheorem{corollary}[theorem]{Corollary}

\newtheorem{definition}[theorem]{Definition}
\newtheorem{proposition}[theorem]{Proposition}

\title{Morse index of min--max stationary integral varifolds}
\author{Mitchell Gaudet, Talant Talipov}
\date{}
\pgfplotsset{compat=1.18}

\begin{document}
\thispagestyle{empty}
\maketitle

\begin{abstract}
    We prove an upper bound for the Morse index of min--max stationary integral varifolds realizing the $d$-dimensional $p$-width of a closed Riemannian manifold.
\end{abstract}

\section{Introduction}

Let $(M^{n},g)$ be a closed Riemannian manifold. The \textit{$d$-dimensional volume spectrum} is the sequence of invariants $\{\omega_p^d(M,g)\}_{p\in\mathbb N}$ introduced by Gromov \cite{Gro06,Gro02,Gro10}. These \textit{$d$-dimensional $p$-widths} serve as a nonlinear analogue of the Laplace spectrum and have driven major progress in the theory of minimal hypersurfaces (see \cite{A1,rigidity,chodosh2020minimal,Dey,gaspar2020index,gaspar2018allen,gaspar2019weyl,irie2018density,yangyangli2, yangyangli1,LMN,marques2017existence,MN1,MarquesNeves,marques2021morse,Pi2,zhou2020multiplicity}). In particular, Yau’s conjecture on the existence of infinitely many closed embedded minimal hypersurfaces was resolved for closed Riemannian manifolds $(M^{n},g)$ with $3\le n \le 7$ \cite{song2018existence}.

When $3\le n\le 7$, each $(n-1)$-dimensional $p$-width equals the weighted area of a smooth, closed, embedded \textit{min--max} minimal hypersurface: there exist pairwise disjoint smooth closed embedded minimal hypersurfaces $\{\Sigma_{p,j}\}_{j=1}^{N(p)}$ and positive integers $\{m_{p,j}\}_{j=1}^{N(p)}$ such that
\begin{equation}\label{eqn:rep_intro}
    \omega_p^{n-1}(M,g)=\sum_{j=1}^{N(p)} m_{p,j}\,\mathrm{Area}_g(\Sigma_{p,j}).
\end{equation}
Moreover, one has the index bound \cite{MarquesNeves}
\begin{equation*}
    \sum_{j=1}^{N(p)} \mathrm{index}(\Sigma_{p,j}) \le p,
\end{equation*}
where $\mathrm{index}(\Sigma)$ denotes the Morse index of $\Sigma$. Higher-dimensional index bounds for minimal surfaces with optimal regularity can be found in \cite{yangyangli}. Although multiplicities greater than one may occur in \eqref{eqn:rep_intro} \cite{WangZhou_mult2}, the multiplicity one conjecture, resolved in \cite{chodosh2020minimal,zhou2020multiplicity}, shows that for a generic metric the hypersurfaces are two-sided and all weights are $1$.
 In that setting one obtains the stronger weighted index bound
\begin{equation*}
    \sum_{\Sigma_{p,j}\ \mathrm{2\!-\!sided}} m_{p,j}\,\mathrm{index}(\Sigma_{p,j})
    \;+\;
    \sum_{\Sigma_{p,j}\ \mathrm{1\!-\!sided}} \frac{m_{p,j}}{2}\,\mathrm{index}(\Sigma_{p,j})
    \;\le\; p .
\end{equation*}

\begin{figure}
    \centering
    \includegraphics[width=4.5in,trim=0 0 0 0,clip]{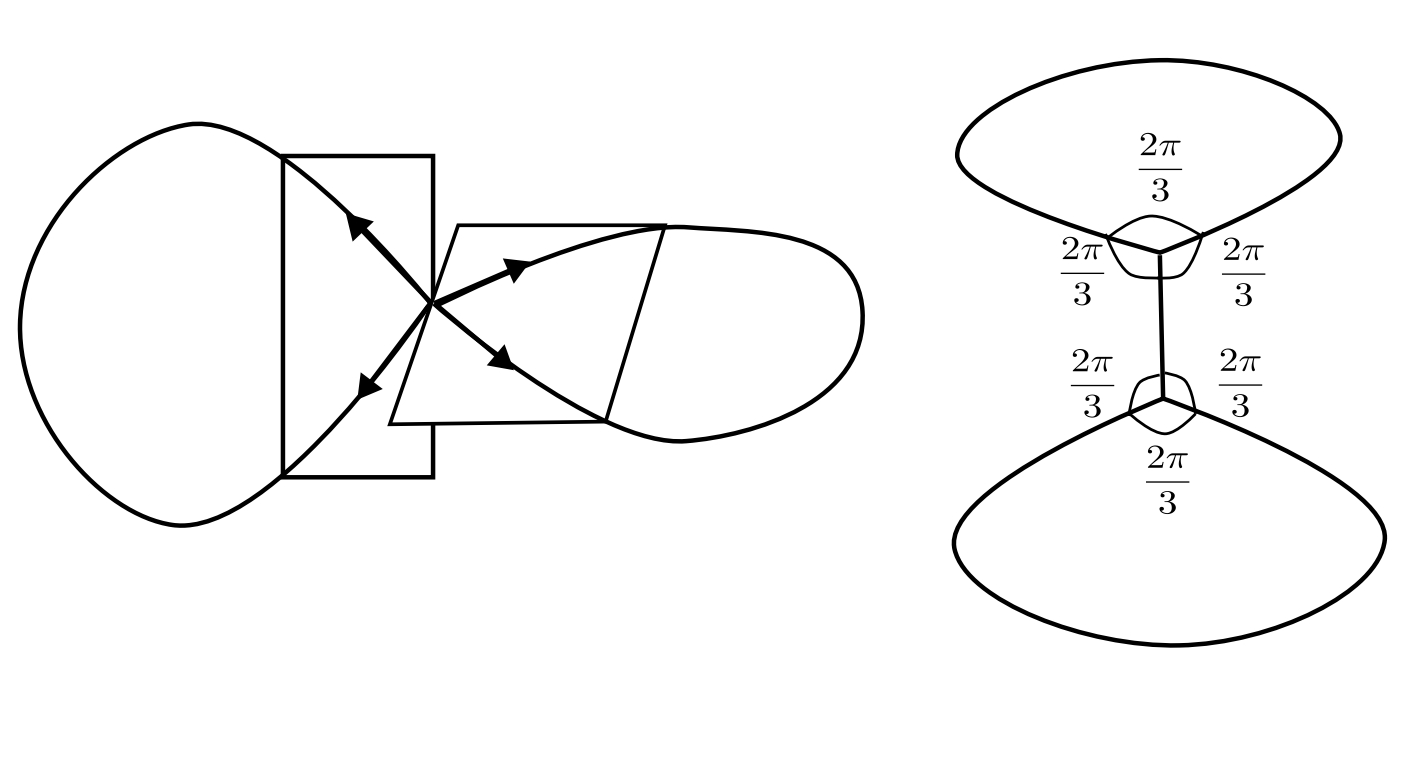}
    \caption{Stationary twisted figure-eight (left) and stationary eyeglass (right)}
    \label{drawing_1}
\end{figure}

Almgren's theory implies that widths are realized by stationary integral varifolds that are almost minimizing in annuli~\cite{A1, A2, MarquesNevesTop}.

For $d = 1$, the relevant min--max objects are \textit{stationary geodesic nets}. A geodesic net is a finite weighted graph immersed in $(M^n,g)$ whose edges are geodesic segments. A geodesic net is \textit{stationary} if it is critical for the length functional, equivalently if the sum of inward unit tangents (with multiplicity) vanishes at every vertex. Any one-dimensional stationary integral varifold is a stationary geodesic net \cite{AllardAlmgren,A2,CaCa,NR,Pi2,Pi1}. Simple examples of stationary geodesic nets in $M^n$ for $n \geq 3$ are illustrated in Figure~\ref{drawing_1}. A stationary twisted figure-eight consists of one vertex and two geodesic loops based at that point. A stationary eyeglass consists of two vertices connected by a geodesic edge, with a geodesic loop attached at each vertex. We refer to \cite{CLNR23, Cheng, HassMorgan,LiStaffa,LiokumovichStaffa,NabutovskyParsch,NabutovskyRotman, Rotman, Staffa_Weyl, Talipov} for developments in the theory of geodesic nets. On closed surfaces, Chodosh–Mantoulidis used the Allen–Cahn min--max framework (with the sine–Gordon potential) to show that the 1-dimensional $p$-widths are realized by unions of closed immersed geodesics:
\begin{theorem}[\cite{chodosh2023p}]
Let $(M^2,g)$ be a closed Riemannian surface. For every $p\in\mathbb N$ there exist closed immersed geodesics $\{\sigma_{p,j}\}_{j=1}^{N(p)}$ and integers $m_{p,j}\ge1$ such that
\[
\omega^1_p(M,g)=\sum_{j=1}^{N(p)} m_{p,j}\,\mathrm{length}_g(\sigma_{p,j}).
\]
\end{theorem}
On surfaces, one also has Morse index control for closed geodesics, see \cite{lorenzo}, which proves that
\[
\sum_{j=1}^{N(p)} \mathrm{index}(\sigma_{p,j}) \le p
\quad\text{and}\quad
\sum_{v\in \mathrm{Vert}(\{\sigma_{p,j}\}_j)} \binom{\mathrm{ord}(v)}{2}\le p .
\]

The main result of this paper is the following.

\begin{theorem}\label{thm:intro-main}
Let $(M^{n},g)$ be a closed Riemannian manifold of dimension $n \geq 2$. For every $p\in\mathbb N$ and $d\in\{1,\ldots,n-1\}$ there exists $S \in \mathcal{IV}_d(M)$ that is stationary and almost minimizing in annuli such that
\[
\omega_p^d(M,g)=\lVert S \rVert(M)
\quad\text{and}\quad
\mathrm{index}(S)\le p(n-d),
\]
where we interpret $\mathrm{index}(S)\le p(n-d)$ as saying $S$ is not $\bigl(p(n-d)+1\bigr)$-unstable (see Definition~\ref{unstable_def} below).
\end{theorem}

The reason why the bound is $p(n-d)$, in contrast to $p$ in the case of hypersurfaces, is the dimension of the cohomology class detected by $p$-sweepouts by $d$-dimensional cycles on an $n$-dimensional manifold, which is exactly $p(n-d)$.

As a corollary of Theorem~\ref{thm:intro-main}, Proposition~\ref{equiv_lemma} and the regularity of 1-dimensional stationary integral varifolds~\cite{AllardAlmgren,A2,CaCa,LiokumovichStaffa,NR,Pi2,Pi1} we obtain the following result.

\begin{theorem}\label{thm:intro-nets}
Let $(M^{n},g)$ be a closed Riemannian manifold of dimension $n \geq 2$. For every $p\in\mathbb N$ there exist disjoint embedded stationary geodesic nets \(G_i:\Gamma_i\to M\), with $\Gamma_i$ a good weighted multigraph for each $1 \leq i \leq N(p)$, such that
\[
\omega_p^1(M,g)= \sum_{i=1}^{N(p)} \operatorname{length}_g(G_i)
\quad\text{and}\quad
\sum_{i=1}^{N(p)}\mathrm{index}(G_i)\le p(n-1).
\]
\end{theorem}




\textbf{Outline.}
The paper is structured as follows. In Section~\ref{prelim}, we recall notation and Almgren--Pitts min--max theory. In Section~\ref{subsec-unstable}, we introduce a notion of $k$-unstable varifold, prove its equivalence with the usual Morse index for stationary geodesic nets, and record some properties. In Section~\ref{subsec-compact}, we establish compactness properties for the space of $k$-unstable stationary integral varifolds. In Section~\ref{subsec-deform}, we prove the Deformation Theorem and Theorem~\ref{thm:intro-main}. In Section~\ref{sec-optimal}, we analyze $\omega^1_1(S^n,g_{\operatorname{round}})$ and prove that the bound must depend on $n$ and that it is optimal for $\omega_1^1$.

We now sketch the proof of Theorem~\ref{thm:intro-main}. In the minimal hypersurface setting, Morse index bounds~\cite{MarquesNeves} use the Structure theorem~\cite{abraham1970bumpy,white2017bumpy,white1991space} to deduce countability of smooth minimal hypersurfaces, which allows one to exclude high-index hypersurfaces inductively, one at a time. In higher codimension the Structure theorem is not available, so we instead use the hierarchical deformation developed in~\cite{yangyangli,song_hierarchical}, adapted to stationary integral varifolds.

Given a min--max sequence for $\omega^d_p$, we homotope its image away from the set of $\bigl(p(n-d)+1\bigr)$-unstable stationary integral $d$-varifolds. The main issue is that a homotopy with this property, while also preserving the min--max property, cannot be constructed in a single step. Instead, since the domains are simplicial complexes, we follow the standard procedure: we first define the homotopy on the $0$-cells and then induct on dimension. To construct the homotopy on higher-dimensional cells, we must ensure that the lower-dimensional homotopies are compatible; for this, we use a construction of Marques--Neves~\cite{MarquesNeves} and then follow an associated gradient flow, which pushes the image away from the $\bigl(p(n-d)+1\bigr)$-unstable varifolds.

\textbf{Acknowledgments.}
The authors are grateful to Prof. Yevgeny Liokumovich for supervision and encouragement. Part of this paper was completed while T.T. was at Nazarbayev University in Astana; T.T. thanks Prof. Durvudkhan Suragan and Makhpal Manarbek for their hospitality. T.T. was supported by the Dr. Sergiy and Tetyana Kryvoruchko Graduate Scholarship in Mathematics. M.G. was supported by the Ontario Graduate Scholarship.

\section{Preliminaries}\label{prelim}
\subsection{Notation}
Let $(M^{n},g)$ be a closed Riemannian manifold. In this section we recall standard notation from geometric measure theory.
\begin{itemize}
    \item $\mathbf{I}_d(M;\mathbb{Z}_2)$: the space of $d$-dimensional mod $2$ flat chains in $M$, equipped with the topology induced by the flat metric $\mathcal{F}$.
    \item $\mathcal{Z}_d(M;\mathbb{Z}_2)$: the space of $d$-dimensional flat cycles $T \in \mathbf{I}_d(M;\mathbb{Z}_2)$ such that $\partial T = 0$.
    \item $\mathbf{M}$: the mass functional on $\mathbf{I}_d(M;\mathbb{Z}_2)$.
    \item $G_d(M):=\{(x,P): x\in M,\ P\in G(T_xM,d)\}$, where $G(T_xM,d)$ is the Grassmannian of unoriented $d$-dimensional subspaces of $T_xM$.
    \item $\mathcal{V}_d(M)$: the space of $d$-varifolds on $M$, i.e. Radon measures on $G_d(M)$, equipped with the weak topology, induced by the varifold metric $\mathbf{F}$ on each $\{V\in\mathcal{V}_d(M):\lVert V \rVert (M) \leq c\}$ for $c>0$.
    \item $\mathcal{IV}_d(M)$: the space of integral $d$-varifolds on $M$.
    \item $\|V\|$: the weight measure on $M$ induced by $V\in\mathcal{V}_d(M)$.
    \item $|T|$: the integral $d$-varifold induced by a mod $2$ flat chain $T\in\mathbf{I}_d(M;\mathbb{Z}_2)$.
    \item The $\mathbf{F}$-metric on $\mathbf{I}_d(M;\mathbb{Z}_2)$ is
    \[
        \mathbf{F}(S,T)\coloneqq \mathcal{F}(S,T)+\mathbf{F}\bigl(|S|,|T|\bigr),\qquad S,T\in\mathbf{I}_d(M;\mathbb{Z}_2).
    \]
    \item $I(1, n)$: the cell complex on the unit interval $I$ whose $1$-cells are the intervals $[0, 1 \cdot 3^{-n}], [1 \cdot 3^{-n}, 2 \cdot 3^{-n}], \cdots, [1 - 3^{-n}, 1]$, and whose $0$-cells are the endpoints $[0], [3^{-n}], [2\cdot 3^{-n}], \cdots, [1]$.
    \item $I(m, n)$: the cell complex on $I^m$, i.e., $I(1, n)^{m\otimes} = \underbrace{I(1, n)\otimes I(1, n)\otimes \cdots \otimes I(1, n)}_{\text{$m$ times}}$.
\end{itemize}

\subsection{Almgren--Pitts min--max theory}

By Almgren's isomorphism theorem \cite{A1,A2,GuL}, the space of mod-\(2\) \(d\)-cycles on the \(n\)-sphere,
\(\mathcal{Z}_d(S^n;\mathbb{Z}_2)\), is weakly homotopy equivalent to the Eilenberg--MacLane space \(K(\mathbb{Z}_2,n-d)\).
Let \(\overline\mu\in H^{\,n-d}\big(\mathcal{Z}_d(S^n;\mathbb{Z}_2);\mathbb{Z}_2\big)\) denote the nontrivial cohomology class.
All cup powers of \(\overline\mu\) are nontrivial, and the cohomology ring of \(\mathcal{Z}_d(S^n;\mathbb{Z}_2)\)
is generated by these cup powers together with their Steenrod squares \cite{Ha}.

Choose a smooth map \(f\colon M\to S^n\) that is a diffeomorphism from a small ball \(B\subset M\)
onto \(S^n\setminus\{p\}\) and that sends \(M\setminus B\) to the point \(p\). The induced map on cycle spaces
\(F\colon\mathcal{Z}_d(M;\mathbb{Z}_2)\to\mathcal{Z}_d(S^n;\mathbb{Z}_2)\) pulls back \(\overline\mu\) to a nontrivial class
\(\mu:=F^*(\overline\mu)\in H^{\,n-d}\big(\mathcal{Z}_d(M;\mathbb{Z}_2);\mathbb{Z}_2\big)\).

Let \(X\) be a finite simplicial complex. A continuous map \(\Phi\colon X\to\mathcal{Z}_d(M;\mathbb{Z}_2)\) is called a \textit{\(p\)-sweepout}
if \(\Phi^*(\mu^p)\neq 0\in H^{\,p(n-d)}(X;\mathbb{Z}_2)\), \(\Phi\) satisfies the
no-concentration-of-mass property (cf. \cite{LMN,MN1}), and \(X\) is a cubical subcomplex of \(I(2p+1,q)\) for some \(q\ge1\).
The collection of all such \(p\)-sweepouts is denoted \(\mathcal{P}^d_p(M,g)\), the \(p\)-\textit{admissible set}. The \(d\)-dimensional \(p\)-width is
\[
\omega_p^d(M,g):=\inf_{\Phi\in\mathcal{P}^d_p}\sup_{x\in X}\mathbf M\bigl(\Phi(x)\bigr).
\]

\begin{definition}
A \textit{min--max sequence} for \(\mathcal{P}^d_p\) is a sequence of \(p\)-sweepouts \(\{\Phi_i\}_{i=1}^\infty\subset\mathcal{P}^d_p\)
such that
\[
\lim_{i\to\infty}\sup_{x\in X_i}\mathbf M\bigl(\Phi_i(x)\bigr)=\omega_p^d(M,g).
\]
Its \textit{critical set} is
\[
\mathbf C\bigl(\{\Phi_i\}\bigr)
:=\Big\{V\in\mathcal V_d(M): \|V\|(M)=\omega_p^d(M,g)\ \text{ and }\ V=\lim_j|\Phi_{i_j}(x_j)|\Big\},
\]
where the limit is taken in the varifold sense.
\end{definition}
We may assume the domains of the min--max sequence are finite simplicial complexes of dimension at most $p(n-d)$ (see Lemma 2.25 in~\cite{Staffa_Weyl}).

Almgren's theory implies that widths are realized by stationary integral varifolds that are almost minimizing in annuli~\cite{A1, A2, MarquesNevesTop}.

\begin{theorem}[Min--max theorem for \(d\)-dimensional widths]
\label{thm:minmax}
For each \(p\) there exists $S \in \mathcal{IV}_d(M)$ that is stationary and almost minimizing in annuli such that
\[
\omega_p^d(M,g)= \| S \|(M).
\]
\end{theorem}

For $d = 1$ one has the
stronger regularity: stationary integral \(1\)-varifolds are stationary geodesic nets \cite{A1,A2,CaCa,NR,Pi2,Pi1}. Combining this regularity with Lemma 2.5 of \cite{LiokumovichStaffa} yields the following.

\begin{theorem}[Min--max theorem for \(1\)-dimensional widths]
\label{thm:1d-minmax}
For each \(p\) there exist disjoint embedded stationary geodesic nets \(G_i:\Gamma_i\to M\), with $\Gamma_i$ a good weighted multigraph for each $1 \leq i \leq N(p)$, such that
\[
\omega_p^1(M,g)= \sum_{i=1}^{N(p)} \operatorname{length}_g(G_i).
\]
\end{theorem}

\begin{definition}
The \textit{Almgren--Pitts realization of the d-dimensional \(p\)-width}, denoted \(\mathcal{APR}_{p,d}(M,g)\), is the set of all \(d\)-varifolds \(S\) such that
\begin{enumerate}
    \item \(\|S\|(M)=\omega_p^d(M,g)\),
    \item \(S\) is a stationary integral varifold that is almost minimizing in annuli.
\end{enumerate}
Theorem~\ref{thm:minmax} guarantees \(\mathcal{APR}_{p,d}(M,g)\neq\emptyset\).
\end{definition}

From now on we fix a closed Riemannian manifold $(M,g)$ so we omit the explicit dependence on $(M,g)$ from the notation.

\section{Index upper bound}
\subsection{Unstable variations and index}\label{subsec-unstable}

We introduce a notion of $k$-unstable stationary integral varifolds from~\cite{MarquesNeves}.

\begin{definition}
    Let $S \in \mathcal{IV}_d(M)$ be a stationary integral varifold and $\varepsilon \geq 0$. We say that $S$ is \textit{$k$-unstable in an $\varepsilon$-neighborhood} if there exist $0 < c_0 < 1$ and a smooth family $\{F_v\}_{v \in \overline{B}^k} \subset \operatorname{Diff}(M)$ with $F_0 = \operatorname{Id}, F_{-v} = F^{-1}_v$ for all $v \in \overline{B}^k$ such that, for any $V \in \overline{\textbf{B}}^{\mathbf{F}}_{2\varepsilon}(S)$, the smooth function
    \begin{equation*}
        A^V: \overline{B}^k \rightarrow [0, +\infty), \quad A^V(v) = \lVert (F_v)_\# V \rVert (M)
    \end{equation*}
    satisfies
    \begin{itemize}
        \item $A^V$ has a unique maximum at $m(V) \in \overline{B}^k_{c_0/\sqrt{10}}(0)$;
        \item $-\frac{1}{c_0} \operatorname{Id} \leq D^2 A^V(u) \leq -c_0 \operatorname{Id}$ for all $u \in \overline{B}^k$.
    \end{itemize}
\end{definition}
Here $(F_v)_\#$ denotes the push-forward operation. Notice that since $A^V$ is strictly concave, the maximum is also the unique critical point. Therefore, since $S$ is stationary, necessarily $m(S)=0$. If $V_i$ tends to $V$ in the $\mathbf{F}$-topology then $A^{V_i}$ tends to $A^V$ in the smooth topology \cite[Section~2.3(2)]{Pi2}.

Notice that if $V_i\rightarrow V$ then, by the prior remarks, $m(V_i)\rightarrow m(V)$. Since the ball $\overline{\mathbf B}^{\mathbf{F}}_{2\varepsilon}(S)$ is metrizable, the map $V \rightarrow m(V)$ is continuous on it.  Thus if $S$ is stationary and $k$-unstable in a $0$-neighborhood then it is $k$-unstable in an $\varepsilon$-neighborhood for some $\varepsilon > 0$.
\begin{definition} \label{unstable_def}
    Let $S \in \mathcal{IV}_d(M)$ be a stationary integral varifold. We say that $S$ is \textit{$k$-unstable} if it is $k$-unstable in an $\varepsilon$-neighborhood for some $\varepsilon > 0$.
\end{definition}

\begin{figure}
  \centering
  \resizebox{0.6\linewidth}{!}{\includegraphics{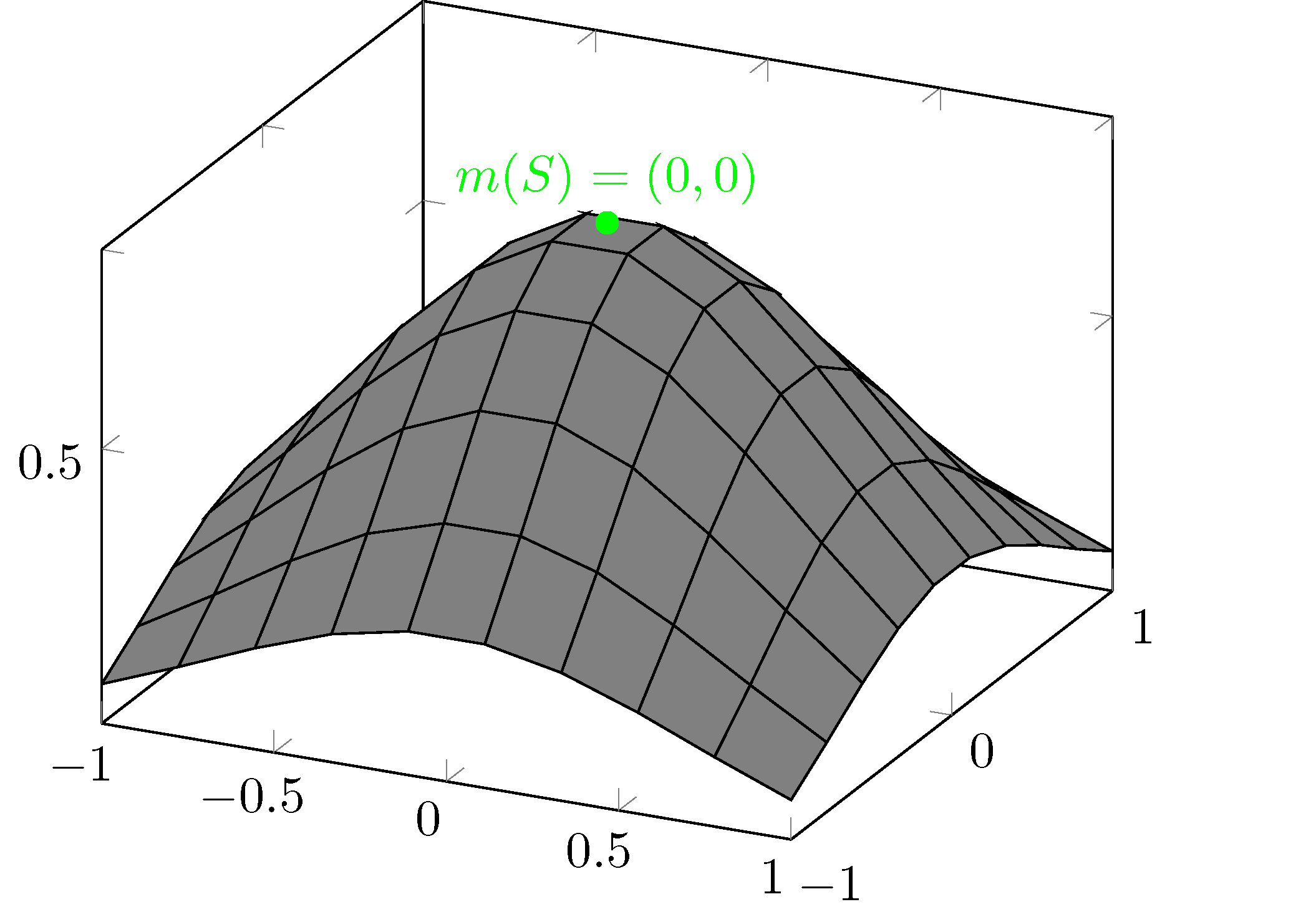}}
  \caption{$A^S$ for stationary integral varifold $S$. Note $m(S)=(0,0)$.}
  \label{graph-1}
\end{figure}

See Figures~\ref{graph-1} and~\ref{graph-2} for the plot of $A^S$ when $S$ is $2$-unstable, and for the plot of $A^V$ for a nonstationary $V\in\overline{\mathbf{B}}^{\mathbf{F}}_{2\varepsilon}(S)$.

\begin{figure}
  \centering
  \resizebox{0.6\linewidth}{!}{\includegraphics{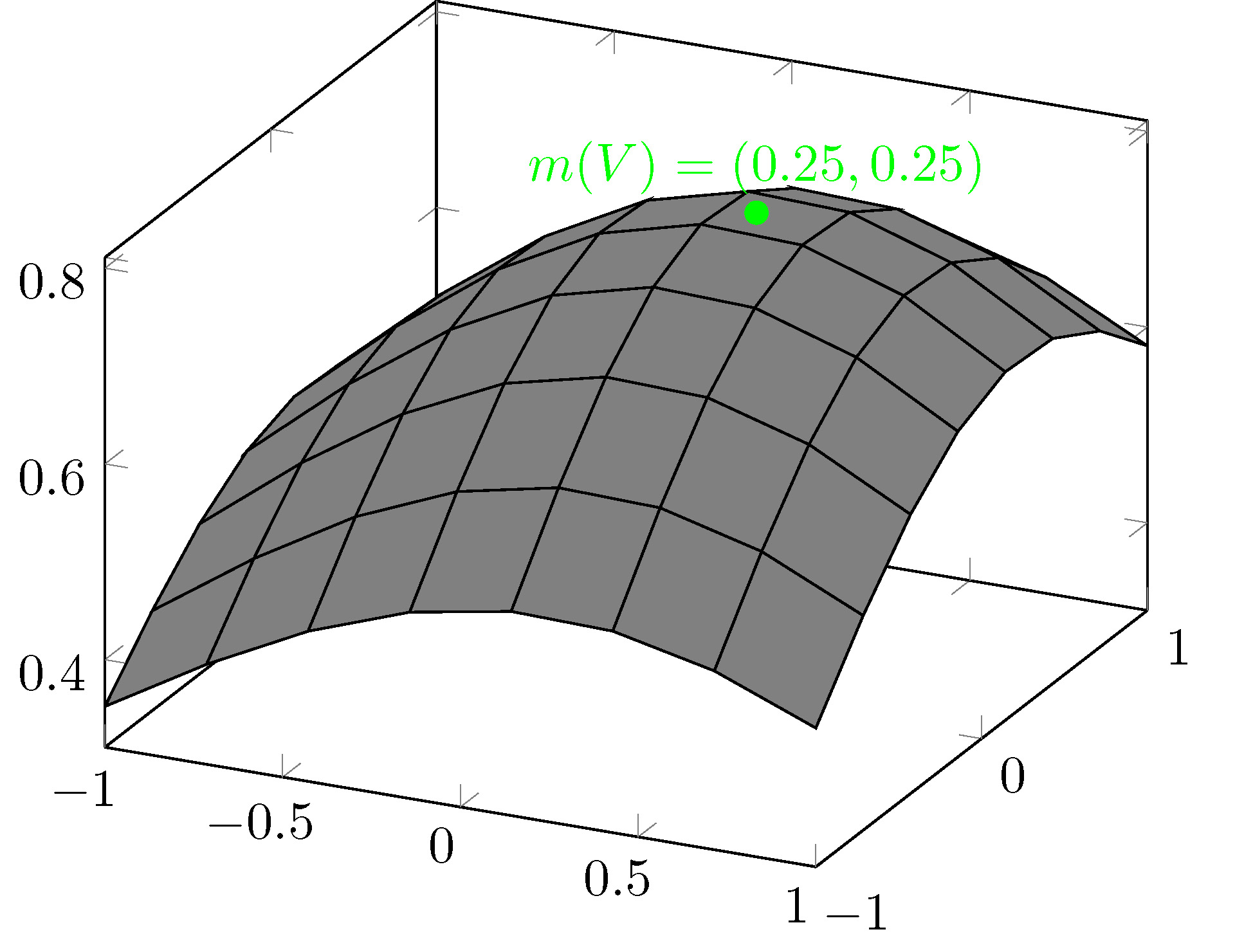}}
  \caption{$A^V$ for $V$ not stationary. Note $m(V)\neq(0,0)$.}
  \label{graph-2}
\end{figure}

Let $G \in \mathcal{IV}_1(M)$ be a stationary geodesic $\Gamma$-net with the set of vertices $\mathscr{V}$. We denote $\mathfrak{X}_k(G)$ the space of continuous vector fields along $G$ which are of class $C^k$ along each edge. Observe that $\mathfrak{X}_k(G)$ is always well-defined for $k\leq2$ and could be defined for larger values of $k$ provided the restrictions of $G$ to the edges have enough regularity. By Lemma~2.5 of \cite{LiokumovichStaffa}, $G$ is equal in the varifold sense to a disjoint union of embedded stationary geodesic nets $G_i:\Gamma_i\to M$, where each $\Gamma_i$ is a good weighted multigraph, for $1\le i\le P$. In what follows, we always regard $G$ as equipped with this net structure.

Since the $G_i$ are disjoint and embedded, every $X\in \mathfrak{X}_2(G)$ can be identified with its pushforward, which is a section of $TM|_{G(\Gamma)}$. We will use the same notation for $X$ and its pushforward.

We define $\mathfrak{C}_{k}(G)\subset \mathfrak{X}_k(G)$ to be the subspace of vector fields with the property that, at each vertex $v\in \mathscr{V}$, they admit a local $C^k$ extension to a neighborhood of $v$. The key point is that if $X\in \mathfrak{C}_{\infty}(G)$, then by Lemma~10.12 in \cite{Lee} there exists a global smooth vector field $\tilde{X}$ on $M$ such that
\[
X=\tilde{X}\big|_{G(\Gamma)}.
\]

Let $Q(\cdot,\cdot)$ be the bilinear form on $\mathfrak{X}_2(G)$ given by the second variation of the length functional at $G$ (see \cite{Staffa} for an explicit formula). Define $\operatorname{index}(G)$ to be the maximal dimension of a subspace of $\mathfrak{C}_{\infty}(G)$ on which $Q$ is negative definite. Note that
\[
\operatorname{index}(G)=\sum_{i=1}^{P}\operatorname{index}(G_i).
\]
We now show that the two notions of instability agree for stationary geodesic nets.

\begin{proposition} \label{equiv_lemma}
    If $G \in \mathcal{IV}_1(M)$ is a stationary geodesic net, then $G$ is $k$-unstable if and only if $\operatorname{index}(G) \geq k$.
\end{proposition}
\begin{proof}
Since
\[
\operatorname{index}(G)=\sum_{i=1}^{P}\operatorname{index}(G_i),
\]
it suffices to prove the statement when $G$ is embedded.

Suppose $\operatorname{index}(G) \geq k$. Then there exist vector fields $X_1, \ldots, X_k \in \mathfrak{C}_{\infty}(G)$ such that
\begin{equation*}
    Q(\sum_i a_i X_i, \sum_i a_i X_i) < 0
\end{equation*}
for any $(a_i) \neq 0 \in \mathbb{R}^k$. For $\delta > 0$ and each $i$ take the extensions $\tilde{X}_i$ as above. By construction, these extensions remain linearly independent on an open neighborhood of $G(\Gamma)$, and moreover $Q$ is negative definite on $\tilde{X}_1, \ldots, \tilde{X}_k$. Denote by $\phi^X_t$ the flow of $X$ and, for some $\varepsilon$ to be chosen later, define
\begin{equation*}
    F_v(x) = F(v, x) = \phi_\varepsilon^{\sum_i v_i \tilde{X}_i} (x)
\end{equation*}
for $v \in \overline{B}^k$ and $x \in M$. Then $F_0 = \operatorname{Id}, F_{-v} = F^{-1}_v$ for all $v \in \overline{B}^k, DA^G(0) = 0$. Moreover, since
\begin{equation*}
    F(tv, x) = \phi_\varepsilon^{t \sum_i v_i \tilde{X}_i} (x) = \phi_t^{\varepsilon \sum_i v_i \tilde{X}_i} (x),
\end{equation*}
we have
\begin{equation*}
    \frac{\partial}{\partial t} F(tv, x) = \varepsilon \sum_i v_i \tilde{X}_i (x).
\end{equation*}
Hence $D^2A^G(0) < 0$. By choosing $\varepsilon$ sufficiently small, we conclude that $G$ is $k$-unstable.

Now suppose $G$ is $k$-unstable, and let $\{F_v\}_{v \in \overline{B}^k}$ be the corresponding family of diffeomorphisms. Define $Y_i(x) = \frac{\partial}{\partial t} \mid_{t = 0} F (t e_i, x) = \frac{\partial F}{\partial v_i} (0, x)$. Then
\begin{equation*}
    0 > \frac{d^2}{dt^2} \mid_{t = 0} \operatorname{length}_g\bigl(F_{tv}(G)\bigr) = Q(\sum_{i} v_i Y_i, \sum_{i} v_i Y_i).
\end{equation*}
This means $\{Y_1, \ldots, Y_k\}$ is linearly independent and spans a subspace of vector fields restricted to which $Q$ is negative definite. Therefore $\operatorname{index}(G) \geq k$.
\end{proof}

\begin{definition}
    For a stationary integral varifold $S \in \mathcal{IV}_d(M)$, we define its \textit{Morse index} to be
    \begin{equation*}
        \operatorname{index}(S) = k,
    \end{equation*}
    provided that it is $k$-unstable but not $(k + 1)$-unstable.
\end{definition}

\begin{remark}
    This definition is compatible with the usual definition of the Morse
index for smooth minimal hypersurfaces (see Proposition 4.3 in~\cite{MarquesNeves}).
\end{remark}

In what follows, we assume that $S \in \mathcal{IV}_d(M)$ is a stationary integral varifold that is $k$-unstable in an $\varepsilon$-neighborhood for some $\varepsilon > 0$. Then, it is associated with a quadruple $(\varepsilon, c_{0}, \{F_v\}, m)$.
\begin{lemma}
\label{lem:nearunstable}
  If $\tilde S\in \mathcal{IV}_d(M)$ is stationary and $\mathbf{F}(S,\tilde S)<\varepsilon$ then $\tilde S$ is also $k$-unstable.
\end{lemma}
\begin{proof}
      Let $\tilde\varepsilon=\frac{1}{2}(\mathbf{F}(S,\tilde S)+\varepsilon)>0$, so that $\mathbf{F}(S,\tilde S)<\tilde\varepsilon<\varepsilon$, and let $\{F_v\}_{v \in \overline{B}^k}$ be the diffeomorphism family from the definition of $S$ being $k$-unstable.

      Then $\overline{\textbf{B}}^{\mathbf{F}}_{2(\varepsilon-\tilde\varepsilon)}(\tilde S)\subset\overline{\textbf{B}}^{\mathbf{F}}_{2\varepsilon}(S)$.
      Hence for any $V\in\overline{\mathbf B}^{\mathbf F}_{2(\varepsilon-\tilde\varepsilon)}(\tilde S)$, $A^V(v) = \lVert (F_v)_\# V \rVert (M)$ also satisfies the two properties. Hence we deduce that $\tilde S$ is also $k$-unstable.
\end{proof}
\begin{lemma}
\label{lem:disjoint}
    There exists $\varepsilon_0 \in (0, \varepsilon)$, such that for any $\tilde{\varepsilon} \in (0, \varepsilon_0)$ and $V \in \mathbf{\overline{B}}^{\textbf{F}}_{2\tilde{\varepsilon}}(S)$,
    \begin{equation*}
        (F_{\mathbb S^{k-1}})_{\#} (V) \cap \mathbf{\overline{B}}^{\textbf{F}}_{2\tilde{\varepsilon}}(S) = \emptyset.
    \end{equation*}
\end{lemma}
\begin{proof}
Suppose not and there exists $V_i \in \mathbf{\overline{B}}^{\mathbf{F}}_{2\varepsilon_i}(S), \varepsilon_i \rightarrow 0$, and $v_i \in \mathbb S^{k-1}$, such that
\begin{align*}
    \mathbf{F}(V_i, S) \leq 2\varepsilon_i \text{ and } (F_{v_i})_{\#}(V_i) \in \mathbf{\overline{B}}^{\textbf{F}}_{2{\varepsilon}_i}(S).
\end{align*}
By the compactness of $\mathbb S^{k-1}$, up to a subsequence, we may assume that $v_i \rightarrow v \in \mathbb S^{k-1}$. Since $F_v$ is smooth, we can take a limit as $i \rightarrow \infty$ and obtain
\begin{equation*}
    (F_v)_{\#}(S) = S.
\end{equation*}
However, by calculating the mass of both sides, we have the following inequality
\begin{align*}
    \|S\|(M) - \|(F_v)_{\#}S\|(M) &= A^S(0) - A^S(v) \\
    &= - \int_0^1 \frac{d}{dt} A^S (tv) dt \\
    &= - \int_{0}^{1} \int_{0}^{t} \frac{d^2}{ds^2} A^S(sv) ds dt \\
    &= - \int_{0}^{1} \int_{0}^{t} D^2 A^S(sv)(v,v)dsdt \\
    &\geq \int_{0}^{1} \int_{0}^{t} c_0 ds dt = \frac{c_0}{2} > 0.
\end{align*}
This implies that $(F_v)_{\#}(S) \neq S$, giving a contradiction.
\end{proof}

In the remainder of this paper, for technical reasons, we always take $\varepsilon = \varepsilon_0$ in the definition of $k$-instability, and it is easy to check that the varifold is still $k$-unstable with the new $\varepsilon$.
\begin{lemma}
\label{lem:grad}
    Given any fixed $V \in \mathbf{\overline{B}}^{\textbf{F}}_{2\varepsilon}(S)$, if an integral curve $x: [0, L] \rightarrow \overline{B}^k$ disjoint from $m(V)$ is defined by
    \begin{equation*}
        \frac{d}{dt} x(t) = - f(t) \nabla A^V\bigl(x(t)\bigr),
    \end{equation*}
    where $f(t)$ is a positive continuous function, then we have
    \begin{equation*}
        A^V\bigl(x(L)\bigr) - A^V\bigl(x(0)\bigr) \leq - \frac{c_0}{2} |x(L) - x(0)|^2.
    \end{equation*}
\end{lemma}
\begin{proof}
By the fundamental theorem of calculus, we have
\begin{align*}
    A^V&\bigl(x(L)\bigr) - A^V\bigl(x(0)\bigr) = \int_0^L \nabla A^V\bigl(x(t)\bigr) \cdot \frac{d}{dt} x(t) dt = - \int_0^L f(t) \Bigl|\nabla A^V\bigl(x(t)\bigr)\Bigr|^2dt \\
    &\leq - \int_0^L f(t) \Bigl|\nabla A^V\bigl(x(t)\bigr)\Bigr| \Bigl(\int_0^t \frac{d}{ds}\Bigl|\nabla A^V\bigl(x(s)\bigr)\Bigr| ds \Bigr)dt \\
    &= - \int_0^L f(t) \Bigl|\nabla A^V\bigl(x(t)\bigr)\Bigr| \Bigl(\int_0^t \frac{\nabla^2 A^V\bigl(x(s)\bigr)\bigl(\nabla A^V\bigl(x(s)\bigr), \frac{d}{ds}x(s)\bigr)}{\Bigl|\nabla A^V\bigl(x(s)\bigr)\Bigr|} ds \Bigr)dt \\
    &= \int_0^L f(t) \Bigl|\nabla A^V\bigl(x(t)\bigr)\Bigr| \Bigl(\int_0^t f(s) \frac{\nabla^2 A^V\bigl(x(s)\bigr)\bigl(\nabla A^V\bigl(x(s)\bigr), \nabla A^V\bigl(x(s)\bigr)\bigr)}{\Bigl|\nabla A^V\bigl(x(s)\bigr)\Bigr|} ds \Bigr)dt \\
    &\leq - c_0 \int_0^L f(t) \Bigl|\nabla A^V\bigl(x(t)\bigr)\Bigr| \Bigl(\int_0^t f(s) \Bigl|\nabla A^V\bigl(x(s)\bigr)\Bigr| ds \Bigr)dt.
\end{align*}
Here, we use the fact that $\nabla A^V(v) \neq 0$ for any $v \neq m(V)$. Indeed, define the path $\gamma: [0,1] \to \overline{B}^{k}$ as
\begin{equation*}
    \gamma(t) = m(V) + t(v - m(V)).
\end{equation*}
Consider the function:
\begin{equation*}
g(t) = \langle \nabla A^V(\gamma(t)), v - m(V) \rangle.
\end{equation*}
By the Fundamental Theorem of Calculus,
\begin{equation*}
g(1) - g(0) = \int_0^1 \frac{d}{dt} g(t)  dt.
\end{equation*}
The derivative is computed using the chain rule:
\begin{align*}
\frac{d}{dt} g(t) 
= \left\langle \frac{d}{dt} \left[ \nabla A^V(\gamma(t)) \right], v - m(V) \right\rangle &= \left\langle D^2 A^V(\gamma(t)) \cdot \dot{\gamma}(t), v - m(V) \right\rangle \\
&= \langle D^2 A^V(\gamma(t)) (v - m(V)), v - m(V) \rangle
\end{align*}
By the $k$-instability condition,
\[
D^2 A^V(u) \leq -c_0 I \quad \forall u \in \overline{B}^k,
\]
which implies
\[
\langle D^2 A^V(\gamma(t)) w, w \rangle \leq -c_0 \lVert w \rVert^2 \quad \forall w \in \mathbb{R}^{k}.
\]
Applying this with $w = v - m(V)$:
\begin{align*}
g(1) - g(0) 
&= \int_0^1 \langle D^2 A^V(\gamma(t)) (v - m(V)), v - m(V) \rangle  dt \\
&\leq \int_0^1 -c_0 \lVert v - m(V) \rVert^2  dt \\
&= -c_0 \lVert v - m(V) \rVert^2.
\end{align*}
Note that
\begin{align*}
g(0) &= \langle \nabla A^V(m(V)), v - m(V) \rangle = 0 \quad \text{(since $\nabla A^V(m(V)) = 0$)} \\
g(1) &= \langle \nabla A^V(v), v - m(V) \rangle
\end{align*}
Thus
\begin{equation*}
\langle \nabla A^V(v), v - m(V) \rangle \leq -c_0 \lVert v - m(V)\rVert^2 < 0 \quad \text{for} \quad v \neq m(V).
\end{equation*}
Now consider the function $l$ satisfying
\begin{equation*}
    l'(t) = \Bigl|\frac{d}{dt} x(t) \Bigr| = f(t) \Bigl|\nabla A^V(x(t))\Bigr|
\end{equation*}
with $l(0) = 0$. Then
\begin{equation*}
    l(L) = \int^L_0 l'(t) dt = \int^L_0 \Bigl|\frac{d}{dt} x(t) \Bigr| dt \geq \Bigl| \int^L_0 \frac{d}{dt} x(t) dt \Bigr| = \Bigl| x(L) - x(0) \Bigr|.
\end{equation*}
Therefore, we have
\begin{align*}
    A^V\bigl(x(L)\bigr) - A^V\bigl(x(0)\bigr) 
    &\leq -c_0 \int^L_0 l'(t) \int_0^t l'(s) ds dt = -c_0 \int_0^L l'(t) l(t) dt \\ 
    &= - \frac{c_0}{2} l(L)^2 \leq - \frac{c_0}{2} \Bigl| x(L) - x(0) \Bigr|^2.
\end{align*}
This completes the proof.
\end{proof}

\begin{lemma} \label{lem:dist_F}
    There exists an increasing function
    \begin{equation*}
        h_S : \mathbb{R}^+ \rightarrow \mathbb{R}^+,
    \end{equation*}
    such that for any $V \in \overline{\mathbf{B}}^{\mathbf{F}}_{2\varepsilon}(S)$, and $v, w \in \overline{B}^k$,
    \begin{equation*}
        |v - w| \geq h_S \Bigl(\mathbf{F}\bigl((F_v)_\#(V), (F_w)_\#(V)\bigr)\Bigr)
    \end{equation*}
\end{lemma}
\begin{proof}
It suffices to show that for every $c > 0$, if
\begin{equation*}
    \Bigl\{ (v,w) : \exists V \in \overline{\mathbf{B}}^{\mathbf{F}}_{2\varepsilon}(S), \mathbf{F}\bigl((F_v)_\#(V), (F_w)_\#(V)\bigr) = c \Bigr\} \neq \emptyset
\end{equation*}
then
\begin{equation*}
    \Bigl\{ | v - w | : \exists V \in \overline{\mathbf{B}}^{\mathbf{F}}_{2\varepsilon}(S), \mathbf{F}\bigl((F_v)_\#(V), (F_w)_\#(V)\bigr) = c \Bigr\}
\end{equation*}
has a positive lower bound. We argue by contradiction. Suppose that there exists a sequence $V_i \in \overline{\mathbf{B}}^{\mathbf{F}}_{2\varepsilon}(S)$ and $v_i, w_i \in \overline{B}^k$ with $\mathbf{F}\bigl((F_{v_i})_\#(V_i), (F_{w_i})_\#(V_i)\bigr) = c$ but $| v_i - w_i | \rightarrow 0$. By compactness, up to a subsequence, we may assume that
\begin{equation*}
    V_i \rightarrow V, v_i \rightarrow v, w_i \rightarrow v,
\end{equation*}
and $\mathbf{F}\bigl((F_v)_\#(V), (F_v)_\#(V)\bigr) = c > 0$, which gives a contradiction.
\end{proof}

\subsection{Compactness under index bounds}\label{subsec-compact}

\begin{definition}
    Let \(\mathcal{U}(\omega^d_p)\) be the subset of \(\mathcal{APR}_{p,d}\) consisting of all \(\bigl(p(n-d)+1\bigr)\)-unstable stationary integral $d$-varifolds, and let \(\mathcal{S}(\omega^d_p)\) denote its complement in \(\mathcal{APR}_{p,d}\), i.e. \(\mathcal{S}(\omega^d_p)\) is the subset of stationary integral $d$-varifolds with Morse index at most \(p(n-d)\).
\end{definition}

\begin{remark}
    If \(\mathcal{U}(\omega^d_p)=\emptyset\), the main theorems are trivial, so we may assume  \(\mathcal{U}(\omega^d_p)\neq\emptyset\).
\end{remark}

\begin{lemma}\label{lem:union_compact}
    The set \(\mathcal{U}(\omega^d_p)\) is a countable union of compact sets, i.e.
    \[
        \mathcal{U}(\omega^d_p)=\bigcup_{i=1}^\infty\mathcal{U}_i(\omega^d_p),
    \]
    where each \(\mathcal{U}_i(\omega^d_p)\) is compact.
\end{lemma}

\begin{proof}
    By the compactness theorem for integral varifolds~\cite{Allard}, \(\mathcal{APR}_{p,d}\) in the set of $d$-varifolds is compact in varifold topology (equivalently, in $\mathbf{F}$-topology as \(\mathcal{APR}_{p,d}\) is uniformly mass bounded). We first show that \(\mathcal{S}(\omega^d_p)\) is closed in \(\mathcal{APR}_{p,d}\), hence compact.

    Let \(\{S_i\}_{i\in\mathbb N}\subset\mathcal{S}(\omega^d_p)\) be a sequence converging to \(S\in\mathcal{APR}_{p,d}\) in the \(\mathbf F\)-topology. We must prove \(\operatorname{index}(S)\le p(n-d)\). Suppose, to the contrary, that \(S\) is \(\bigl(p(n-d)+1\bigr)\)-unstable in an $\varepsilon$-neighborhood, with $\varepsilon>0$. Choose \(i_0\) large enough that \(\mathbf F(S,S_i)<\varepsilon\) for all \(i\ge i_0\). For such \(i\), \(S_i\) is \(\bigl(p(n-d)+1\bigr)\)-unstable by Lemma~\ref{lem:nearunstable}, contradicting \(S_i\in\mathcal{S}(\omega^d_p)\). Therefore \(S\) cannot be \(\bigl(p(n-d)+1\bigr)\)-unstable, so \(\operatorname{index}(S)\le p(n-d)\) and \(S\in\mathcal{S}(\omega^d_p)\). This proves \(\mathcal{S}(\omega^d_p)\) is closed.

    If \(\mathcal{S}(\omega^d_p)=\emptyset\), then \(\mathcal{U}(\omega^d_p)=\mathcal{APR}_{p,d}\) is compact and the lemma follows. Otherwise, for each \(i\in\mathbb N\) define
    \[
        \mathcal{U}_i(\omega^d_p):=\{S\in\mathcal{U}(\omega^d_p):\ \mathbf F\bigl(S,\mathcal{S}(\omega^d_p)\bigr)\ge 1/i\}.
    \]
    Each \(\mathcal{U}_i(\omega^d_p)\) is closed in the compact set \(\mathcal{APR}_{p,d}\), hence compact. Since every \(S\in\mathcal{U}(\omega^d_p)\) has positive \(\mathbf F\)-distance to the closed set \(\mathcal{S}(\omega^d_p)\), we have
    \[
        \mathcal{U}(\omega^d_p)=\bigcup_{i=1}^\infty\mathcal{U}_i(\omega^d_p),
    \]
    as required.
\end{proof}

\subsection{Deformation theorem}\label{subsec-deform}

Fix a sequence of stationary integral varifolds \(\{S_k\}_{k=1}^\infty \subset \mathcal{U}(\omega^d_p)\). By the definition of \(\bigl(p(n-d)+1\bigr)\)-instability, each \(S_k\) is associated with a quadruple \((\varepsilon_k, c_{0,k}, \{F^k_v\}, m_k)\).

For a nonempty set \(K\subset \mathbb{N}^+\), define
\[
\mathbf{B}^{\mathbf F}_K
:= \mathbf{B}^{\mathbf F}_K\bigl(\{S_k\},\{\varepsilon_k\}\bigr)
= \bigcap_{k\in K}\mathbf{B}^{\mathbf F}_{\varepsilon_k}(S_k),
\qquad
\bar{k}(K):=\min\{k\in K\}.
\]
For \(\lambda\ge 0\), set
\[
\mathbf{B}^{\mathbf F}_{\lambda,K}
:= \bigcap_{k\in K}\mathbf{B}^{\mathbf F}_{(2-2^{-\lambda})\varepsilon_k}(S_k).
\]
In particular, for any \(0\le \lambda_1\le \lambda_2\),
\[
\mathbf{B}^{\mathbf F}_K=\mathbf{B}^{\mathbf F}_{0,K}
\subset \mathbf{B}^{\mathbf F}_{\lambda_1,K}
\subset \mathbf{B}^{\mathbf F}_{\lambda_2,K}
\subset \bigcap_{k\in K}\mathbf{B}^{\mathbf F}_{2\varepsilon_k}(S_k).
\]

Let \(K\subset \mathbb{N}^+\) be nonempty and finite, \(\lambda\ge 0\), and assume \(\mathbf{B}^{\mathbf F}_{\lambda,K}\neq\emptyset\). For each
\[
V\in \overline{\mathbf{B}}^{\mathbf F}_{\lambda,K}
:= \bigcap_{\lambda'>\lambda}\mathbf{B}^{\mathbf F}_{\lambda',K},
\]
define the gradient flow \(\{\phi^V_{\lambda,K}(\cdot,t)\}_{t\ge 0}\subset\mathrm{Diff}(\overline{B}^{p(n-d)+1})\) generated by the vector field
\[
u \longmapsto
-\,\mathbf{F}\!\left((F^{\bar{k}(K)}_{u})_{\#}(V),\ (\mathbf{B}^{\mathbf F}_{\lambda+1,K})^c\right)\,\nabla A^V_{\bar{k}(K)}(u),
\]
where \(A^V_{\bar{k}(K)}(u):=\bigl\|(F^{\bar{k}(K)}_{u})_{\#}(V)\bigr\|(M)\) for \(u\in\overline{B}^{p(n-d)+1}\).

We now prove two lemmas. The first asserts that the flow decreases mass \emph{uniformly}, provided that we maintain a uniform distance from the set where \(\nabla A^V_{\bar{k}(K)}=0\).

\begin{lemma}\label{lem:gradflow}
Let \(\lambda>0\) and \(K\subset\mathbb N^+\) be finite, and take \(\eta>0\) sufficiently small. Then there exist constants \(c(\lambda,K)>0\) and \(T>0\) such that the following holds: for any \(V\in\overline{\mathbf{B}}^{\mathbf F}_{\lambda,K}\) and any \(v\in\overline{B}^{p(n-d)+1}\) with \(|v-m_{\bar{k}(K)}(V)|\ge \eta\) and
\[
(F_v^{\bar{k}(K)})_\#(V)\in \overline{\mathbf{B}}^{\mathbf F}_{\lambda+0.5,K},
\]
we have
\[
A^V_{\bar{k}(K)}\bigl(\phi^V_{\lambda,K}(v,T)\bigr)
<
A^V_{\bar{k}(K)}\bigl(\phi^V_{\lambda,K}(v,0)\bigr)-c(\lambda,K).
\]
\end{lemma}

\begin{proof}
We may assume \(\overline{\mathbf{B}}^{\mathbf F}_{\lambda,K}\neq\emptyset\), otherwise the statement is trivial. Note that \(\overline{\mathbf{B}}^{\mathbf F}_{\lambda,K}\) is compact.

For \(l\in\mathbb R\), set
\[
D_l(V):=\bigl\{\,v\in\overline{B}^{p(n-d)+1}:\ (F_v^{\bar{k}(K)})_\#(V)\in \overline{\mathbf{B}}^{\mathbf F}_{\lambda+l,K}\,\bigr\}
\]
and
\[
d(V):=\operatorname{dist}_{\mathrm{Euc}}\bigl(D_{0.5}(V),\,\partial D_1(V)\bigr).
\]
Each \(D_l(V)\) is closed for \(l>0\) because \(v\mapsto (F_v^{\bar{k}(K)})_\#(V)\) is continuous in the varifold topology.

\textbf{Step 1.} We show \(d(V)\) has a uniform positive lower bound.

\emph{Part 1: \(d(V)>0\).}  
If \(w\in\partial D_1(V)\), then by continuity either \(w\in\partial\overline{B}^{p(n-d)+1}=\mathbb S^{p(n-d)}\) or \((F_w^{\bar{k}(K)})_\#(V)\in \partial\overline{\mathbf{B}}^{\mathbf F}_{\lambda+1,K}\). The first case is excluded by Lemma~\ref{lem:disjoint}. In the second case, since \(\partial\overline{\mathbf{B}}^{\mathbf F}_{\lambda+1,K}\cap \overline{\mathbf{B}}^{\mathbf F}_{\lambda+0.5,K}=\emptyset\), we must have \(w\notin D_{0.5}(V)\). Thus \(d(V)>0\).

\emph{Part 2: Uniform positivity.}  
By compactness there is a uniform lower bound: if not, there would exist \(V_i\to V\) in \(\overline{\mathbf{B}}^{\mathbf F}_{\lambda+0.5,K}\) with \(d(V_i)\to0\). Compactness yields points \(v_{i,0.5}\in D_{0.5}(V_i)\) and \(v_{i,1}\in\partial D_1(V_i)\) achieving \(d(V_i)=|v_{i,0.5}-v_{i,1}|\). Passing to a subsequence, \(v_{i,0.5},v_{i,1}\to v\in\overline{B}^{p(n-d)+1}\). Since \(D_l(V_i)\) and \(\partial D_l(V_i)\) are closed for \(l>0\), we get \(v\in D_{0.5}(V)\cap\partial D_1(V)\), hence \(d(V)=0\), a contradiction. Therefore \(d_0:=\inf_{V}d(V)>0\).

\textbf{Step 2.} Define
\[
P:=\bigl\{(V,v)\in \overline{\mathbf{B}}^{\mathbf F}_{\lambda,K}\times \overline{B}^{p(n-d)+1}:\ v\neq m_{\bar{k}(K)}(V),\ v\in D_{0.5}(V)\bigr\}.
\]
For any \((V,v)\in P\), the function \(t\mapsto A^V_{\bar{k}(K)}\bigl(\phi^V_{\lambda,K}(v,t)\bigr)\) is strictly decreasing for \(t>0\). Set
\[
\delta A_{\lambda,K}(V,v)
:=\lim_{t\to\infty}\bigl(A^V_{\bar{k}(K)}(\phi^V_{\lambda,K}(v,0))-
A^V_{\bar{k}(K)}(\phi^V_{\lambda,K}(v,t))\bigr)>0.
\]
We claim there exists \(c=c(\lambda,K)>0\) such that \(\delta A_{\lambda,K}(V,v)\ge 2c\) for all \((V,v)\in P\). Indeed, by Lemma~\ref{lem:grad},
\[
A^V_{\bar{k}(K)}(\phi^V_{\lambda,K}(v,0))-
A^V_{\bar{k}(K)}(\phi^V_{\lambda,K}(v,t))
\ \ge\ \frac{c_{0,\bar{k}(K)}}{2}\,\bigl|\phi^V_{\lambda,K}(v,0)-\phi^V_{\lambda,K}(v,t)\bigr|^2.
\]
Since \(\phi^V_{\lambda,K}(v,t)\) approaches \(\partial D_{\mathrm{1}}(V)\) as \(t\to\infty\), we obtain
\[
\delta A_{\lambda,K}(V,v)\ \ge\ d(V)^2\ \ge\ d_0^2>0.
\]

\textbf{Step 3.} It remains to find \(T=T(\eta,K,\lambda)>0\) so that for all \((V,v)\in P\) with \(|v-m_{\bar{k}(K)}(V)|\ge \eta\),
\[
A^V_{\bar{k}(K)}\bigl(\phi^V_{\lambda,K}(v,T)\bigr)
<
A^V_{\bar{k}(K)}\bigl(\phi^V_{\lambda,K}(v,0)\bigr)-c(\lambda,K).
\]
Arguing by contradiction as in \cite[Lemma~4.5]{MarquesNeves}, assume there exist \((V_i,v_i)\in P\) with \(|v_i-m_{\bar{k}(K)}(V_i)|\ge \eta\) such that
\[
A^{V_i}_{\bar{k}(K)}\bigl(\phi^{V_i}_{\lambda,K}(v_i,t)\bigr)
\ \ge\
A^{V_i}_{\bar{k}(K)}\bigl(\phi^{V_i}_{\lambda,K}(v_i,0)\bigr)-c(\lambda,K)
\quad\text{for all }t\in[0,i].
\]
By compactness, \(V_i\to V\) and \(v_i\to v\) with \((V,v)\in P\) and \(|v-m_{\bar{k}(K)}(V)|\ge \eta\). Passing to the limit yields
\[
A^{V}_{\bar{k}(K)}\bigl(\phi^{V}_{\lambda,K}(v,t)\bigr)
\ \ge\
A^{V}_{\bar{k}(K)}\bigl(\phi^{V}_{\lambda,K}(v,0)\bigr)-c(\lambda,K)
\quad\text{for all }t\ge 0,
\]
which contradicts Step~2. The lemma follows.
\end{proof}

Of course, we need a way to perturb ``bad'' $d$-varifolds away from the critical set $m_{\bar k (K)}(V)$. We will use a construction of Marques--Neves to do this.

Notice that, while defining and proving the perturbation, our setting is slightly different from that of \cite[Theorem~5.1]{MarquesNeves}. We shall prove it in sufficient generality to encompass both situations.

The basic topological idea is the following. We are given a map $\theta:X^l\rightarrow B^{N+1}$ from an $l$-dimensional simplicial complex to a $N+1$-dimensional ball, which we want to uniformly avoid, for some positive integers $l$ and $N$. Since the ball is convex, we can define a homotopy between the constant $0$ map and any other map into $B^{N+1}$ by linear interpolation up to time $t=1$.

In defining the homotopy at $t=1$, we need to construct a map $b:X^l\rightarrow B^{N+1}$ that maps each face $f$ into a small $N+1$-dimensional ball $B^{N+1}(p_f)$ centered at a point $p_f$, possibly with the  point $\{0\}$ removed. As is standard in this context, we construct $b$ by induction over the skeleta of increasing dimension. Recall from \cite[Theorem~7.1]{DavisKirk} that we can extend this function continuously from the boundary of a $(j+1)$-face $e$ to $e$ if and only if the corresponding obstruction cocycle $\alpha\in C^{j+1}\bigl(e,\pi_j(B^{N+1}(p_e)\setminus\{0\})\bigr)$ vanishes.

To guarantee this, notice that $\pi_j(B^{N+1}(p_e)\setminus\{0\})$ is either equal to $0$ (if $0\not\in B^{N+1}(p_e)$) or to $\pi_j(\mathbb S^{N})$ otherwise. Therefore, if $j<N$ all obstruction cocycles vanish, and so we may extend.

\begin{lemma}\label{unsthomotopy}
Let $N>0$ be an integer, let $V\in\mathcal{IV}_d(M)$ be a varifold, and let $\varepsilon>0$. Suppose we are given
\begin{enumerate}
    \item $l\in\{0,\dots,N\}$ and $\delta>0$;
    \item an $l$-dimensional finite simplicial complex $X^l$;
    \item a continuous map $\Psi:X^l\rightarrow \overline{\mathbf{B}}^\mathbf{F}_{2\varepsilon}(V)$;
    \item a continuous map $\phi:\overline{\mathbf{B}}^\mathbf{F}_{2\varepsilon}(V)\rightarrow \overline{B}^{N+1}$.
\end{enumerate}
Then there exists a homotopy $\hat H:X^l\times[0,1]\rightarrow \overline B^{N+1}_\delta$ such that
\begin{enumerate}
    \item $\hat H(x,0)=0$ for all $x\in X^l$;
    \item $\inf_{x\in X^l}\bigl|\phi(\Psi(x))-\hat H(x,1)\bigr|\geq \eta>0$
\end{enumerate}
for some constant $\eta>0$.
\end{lemma}
\begin{remark}
    This lemma could actually be stated without referring to varifolds at all, and only in terms of continuous maps from $l$-dimensional simplicial complexes to $B^{i}$ with $i>l$ (see the discussion prior to the lemma). However, the situation we care about is when $\phi=m_{\bar k(K)}$, so it is convenient to phrase it in this way.
\end{remark}
\begin{proof}
    We start with the same reduction as in \cite[Theorem~5.1]{MarquesNeves}: by repeated barycentric subdivision of the finite simplicial complex and by the uniform continuity of $\phi\circ\Psi$ and $\Psi$, we may assume that, for each closed face $e\subseteq X^l$, it holds that $\operatorname{diam}_\mathbf{F}\bigl(\Psi(e)\bigr)<\alpha$ and, if $x,y$ lie in the same face of $X^l$, then
    \begin{equation*}
        \bigl|\phi(\Psi(x))-\phi(\Psi(y))\bigr|<\alpha.
    \end{equation*}
    Here $\alpha=\min\{\delta 2^{-(N+2)},\,\varepsilon/4\}$.

    Now let $(X^l)^{(i)}$ denote the $i$-skeleton. It is clear that, if we define a map $a:X^l\rightarrow B^{N+1}_\delta$ satisfying
    \[
        \inf_{x\in X^l}\bigl|\phi(\Psi(x))-a(x)\bigr|\geq \eta>0,
    \]
    then we can define the required homotopy by setting $\hat H(x,t)=t\,a(x)$.

    We also notice that it is enough to define a function $b:X^l\rightarrow\mathbb R^{N+1}\setminus\{0\}$ such that
    \[
        b(x)\in B^{N+1}_\delta\bigl(-\phi(\Psi(x))\bigr),
    \]
    and then define $a=b+\phi\circ\Psi$. We will define this $b$ inductively on the skeleta, as we will do later in the proof of the Deformation theorem.

    Clearly, we can define
    \[
        b:(X^l)^{(0)}\rightarrow B^{N+1}_{\delta 2^{-N}}\bigl(-\phi(\Psi(x))\bigr)\setminus\{0\}
    \]
    on the $0$-skeleton.

    Suppose we have defined
    \[
        b:(X^l)^{(j)}\rightarrow B^{N+1}_{\delta 2^{-N+j}}\bigl(-\phi(\Psi(x))\bigr)\setminus\{0\}
    \]
    for some $j\in\{0,\dots,N-1\}$; we will construct it on the $(j+1)$-skeleton. Pick a $j$-face $e$ of $X^l$ with center $y_e$. Then, by construction,
    \[
        b:\partial e\rightarrow B^{N+1}_{\delta 2^{-N+j}+\alpha}\bigl(-\phi(\Psi(y_e))\bigr)\setminus\{0\}.
    \]

    Recall that the homotopy group $\pi_j(\mathbb S^N)=0$ since $j<N$. Therefore, all obstruction cocycles to extending $b$ to a map
    \[
        b:e\rightarrow B^{N+1}_{\delta 2^{-N+j}+\alpha}\bigl(-\phi(\Psi(y_e))\bigr)\setminus\{0\}
    \]
    vanish. Once again, by our subdivision we obtain
    \[
        b(x)\in B^{N+1}_{\delta 2^{-N+j}+2\alpha}\bigl(-\phi(\Psi(x))\bigr)\setminus\{0\},
    \]
    but by our choice of parameters we have $\delta 2^{-N+j}+2\alpha<\delta 2^{-N+j+1}$, and so the induction closes.
\end{proof}
As an immediate corollary we have the following.

\begin{corollary}\label{unsthomotopycrit}
Let \(\lambda>0\), and let \(K\subset \mathbb{N}^+\) be a nonempty finite set with \(\mathbf{B}^\mathbf{F}_{\lambda,K}\neq\emptyset\). Suppose also that we are given
\begin{enumerate}
    \item $l\in\{0,\dots,p(n-d)\}$ and $\delta>0$;
    \item an $l$-dimensional finite simplicial complex $X^l$;
    \item a continuous map $\Theta:X^l\rightarrow \overline{\mathbf{B}}^\mathbf{F}_{2\varepsilon_{\bar k(K)}}(S_{\bar k(K)})$.
\end{enumerate}
Then there exists a homotopy $\hat H:X^l\times[0,1]\rightarrow \overline B^{p(n-d)+1}_\delta$ such that
\begin{enumerate}
    \item $\hat H(x,0)=0$ for all $x\in X^l$;
    \item $\inf_{x\in X^l}\bigl|m_{\bar k(K)}(\Theta(x))-\hat H(x,1)\bigr|\geq \eta>0$
\end{enumerate}
for some constant $\eta>0$.
\end{corollary}

We use this as a perturbation away from the critical set. Combining the two lemmas, we obtain:

\begin{lemma}\label{lem:homotopy}
Let \(\lambda>0\), and let \(K\subset \mathbb{N}^+\) be a nonempty finite set with \(\mathbf{B}^\mathbf{F}_{\lambda,K}\neq\emptyset\). Let \(c(\lambda,K)\) be the constant from Lemma~\ref{lem:gradflow}. For any \(\delta\in(0,c(\lambda,K)/4)\), any \(l\in\{0,1,\dots,p(n-d)\}\), and any continuous map
\[
\Theta:X^{l}\to \mathbf{B}^\mathbf{F}_{\lambda,K},
\]
where \(X^l\) is an \(l\)-dimensional finite simplicial complex, there exists a homotopy
\[
H^{\Theta,\delta}_{\lambda,K}:X^l\times[0,1]\to \mathbf{B}^\mathbf{F}_{\lambda+1,K}
\]
such that:
\begin{enumerate}
    \item \(H^{\Theta,\delta}_{\lambda,K}(\cdot,0)=\Theta(\cdot)\).
    \item For all \(x\in X^l\), \(\ \|H^{\Theta,\delta}_{\lambda,K}(x,1)\|(M)\le \|\Theta(x)\|(M)-\tilde c(\lambda,K)\).
    \item For all \(x\in X^l\) and \(t\in[0,1]\), \(\ \|H^{\Theta,\delta}_{\lambda,K}(x,t)\|(M)-\|\Theta(x)\|(M)\le \delta\).
\end{enumerate}
Here \(\tilde c(\lambda,K):=c(\lambda,K)/2\). More precisely:
\begin{itemize}
    \item For \(t\in[0,1/2]\), \(\ \mathbf{F}\bigl(H^{\Theta,\delta}_{\lambda,K}(x,t),\Theta(x)\bigr)\le \delta\).
    \item For \(t\in[1/2,1]\) there exists \(v(x)\in \overline{B}^{p(n-d)+1}\) such that
    \[
        H^{\Theta,\delta}_{\lambda,K}(x,t)=\phi^{\Theta(x)}_{\lambda,K}\bigl(v(x),(2t-1)T\bigr).
    \]
\end{itemize}
\end{lemma}

\begin{proof}
First, since \(\mathbf{F}\bigl(\overline{\mathbf{B}}^\mathbf{F}_{\lambda,K},\,\partial \overline{\mathbf{B}}^\mathbf{F}_{\lambda+0.5,K}\bigr)>0\), there exists \(\delta_0\in(0,\delta)\) such that for any \(V\in \overline{\mathbf{B}}^\mathbf{F}_{\lambda,K}\) and any varifold \(W\) with \(\mathbf{F}(V,W)\le \delta_0\), we have
\[
\|W\|(M)\le \|V\|(M)+\delta
\quad\text{and}\quad
W\in \overline{\mathbf{B}}^\mathbf{F}_{\lambda+0.5,K}.
\]
By compactness of \(\overline{\mathbf{B}}^\mathbf{F}_{\lambda,K}\) and \(\overline{B}^{p(n-d)+1}\), there exists \(\delta_1>0\) such that
\[
\mathbf{F}\bigl((F^{\bar k(K)}_{v})_\#(V),\,V\bigr)\le \delta_0<\delta
\]
for all \(V\in \overline{\mathbf{B}}^\mathbf{F}_{\lambda,K}\) and all \(v\in \overline{B}^{p(n-d)+1}(\delta_1)\).

Since for every \(x\in X^l\), \(\Theta(x)\in \overline{\mathbf{B}}^\mathbf{F}_{\lambda,K}\subset \overline{\mathbf{B}}^\mathbf{F}_{2\varepsilon_{\bar k(K)}}(S_{\bar k(K)})\), we may use Corollary~\ref{unsthomotopycrit} to obtain a homotopy
\[
\widehat H:X^l\times[0,1]\to \overline{B}^{p(n-d)+1}(\delta_1),
\]
with \(\widehat H(\cdot,0)=0\) and
\[
\inf_{x\in X^l}\bigl|\,m_{\bar k(K)}(\Theta(x))-\widehat H(x,1)\,\bigr|\ \ge \ \eta>0,
\]
for some \(\eta=\eta(\delta,K,\Theta,\lambda)>0\). This homotopy uniformly avoids the set where \(\nabla A^{V}_{\bar k(K)}=0\), while remaining close to the initial point.

Finally, applying Lemma~\ref{lem:gradflow} with \(v=\widehat H(x,1)\) for \(\Theta(x)\), define
\[
H^{\Theta,\delta}_{\lambda,K}(x,t)=
\begin{cases}
\bigl(F^{\bar k(K)}_{\widehat H(x,2t)}\bigr)_{\#}\,\Theta(x), & t\in[0,1/2],\\[2pt]
\phi^{\Theta(x)}_{\lambda,K}\bigl(\widehat H(x,1),(2t-1)T\bigr), & t\in(1/2,1].
\end{cases}
\]
By the choices of \(\delta_0,\delta_1\) and the conclusion of Lemma~\ref{lem:gradflow}, this homotopy satisfies all three properties stated above.
\end{proof}

\begin{theorem}[Deformation Theorem]\label{Thm:deform}
Let \(\{\Phi_i\}_{i\in\mathbb N}\) be a min--max sequence for \(\omega_p^d\), and let \(X_i=\operatorname{dmn}(\Phi_i)\) be its domain with \(\dim X_i\le p(n-d)\). After passing to a subsequence, there exists a sequence \(\{\Psi_i\}_{i\in\mathbb N}\) such that:
\begin{enumerate}
    \item \(\Psi_i\) is homotopic to \(\Phi_i\) in the \(\mathbf F\)-topology;
    \item \(\mathbf L\bigl(\{\Psi_i\}_{i\in\mathbb N}\bigr)=\omega^d_p\);
    \item There exists a function \(\bar\varepsilon:\mathcal U(\omega^d_p)\to(0,\infty)\) with the following property: for every \(S\in\mathcal U(\omega^d_p)\) there exists \(i_0\in\mathbb N\) such that, for all \(i\ge i_0\),
    \[
        \bigl|\Psi_i\bigr|(X_i)\ \cap\ \mathbf B^{\mathbf F}_{\bar\varepsilon(S)}(S)\;=\;\emptyset .
    \]
    In particular, \(\mathbf C(\{\Psi_i\})\cap \mathcal{APR}_{p,d} \subset \mathcal S(\omega^d_p)\).
\end{enumerate}
\end{theorem}
\begin{proof}
The proof consists of 4 steps.

\textbf{Step 1.} We construct a (finite or countable) sequence of \(\bigl(p(n-d)+1\bigr)\)-unstable stationary integral varifolds
\(\{S_k\}\subset\mathcal{U}(\omega^d_p)\), together with associated quadruples
\(\{(\varepsilon_k,c_{0,k},\{F^k_v\},m_k)\}\), so that the following hold:

\begin{itemize}
    \item For \(k_1\neq k_2\) we have \(\mathbf B^\mathbf F_{\varepsilon_{k_1}}(S_{k_1})\not\subset \mathbf B^\mathbf F_{\varepsilon_{k_2}}(S_{k_2})\).
    \item \(\mathcal{U}(\omega^d_p)\subset\bigcup_{k=1}^\infty \mathbf B^\mathbf F_{\varepsilon_k}(S_k)\), and moreover, for every \(m\in\mathbb N\) there exists \(k_m\in\mathbb N\) with
    \[
        \mathcal{U}_m(\omega^d_p)\subset\bigcup_{i=1}^{k_m}\mathbf B^\mathbf F_{\varepsilon_i}(S_i).
    \]
    \item For each fixed \(k\), the set
    \[
        \bar K(S_k):=\{k':\ \mathbf B^\mathbf F_{2\varepsilon_{k'}}(S_{k'})\cap\mathbf B^\mathbf F_{2\varepsilon_k}(S_k)\neq\emptyset\}
    \]
    is finite.
\end{itemize}

Indeed, for every \(S\in\mathcal{U}(\omega^d_p)\) choose a quadruple \((\varepsilon_S,c_{0,S},\{F^S_v\},m_S)\). The family \(\{\mathbf B^\mathbf F_{\varepsilon_S}(S)\}_{S\in\mathcal{U}(\omega^d_p)}\) covers \(\mathcal{U}(\omega^d_p)\). Note that one may always replace \(\varepsilon_S\) by a smaller positive number and keep \(S\) \(\bigl(p(n-d)+1\bigr)\)-unstable with the new quadruple.

If \(\mathcal{U}(\omega^d_p)\) is compact, we extract a finite subcovering \(\{\mathbf B^\mathbf F_{\varepsilon_k}(S_k)\}_{k=1}^{k_1}\) and take the corresponding quadruples. Then all the required properties hold.

Otherwise, by Lemma~\ref{lem:union_compact} we have a decomposition
\[
    \mathcal{U}(\omega^d_p)=\bigcup_{m\in\mathbb N}\mathcal{U}_m(\omega^d_p),
\qquad
\mathcal{U}_m(\omega^d_p):=\{S\in\mathcal{U}(\omega^d_p):\ \mathbf F(S,\mathcal{S}(\omega^d_p))\ge 1/m\},
\]
with each \(\mathcal{U}_m(\omega^d_p)\) compact. We construct the sequence \(\{S_k\}\) inductively in \(m\).

For \(m=1\): for every \(S\in\mathcal{U}_1(\omega^d_p)\) choose \(\varepsilon_S\) so that \(\varepsilon_S<1/4\). Extract a finite subcovering \(\{\mathbf B^\mathbf F_{\varepsilon_k}(S_k)\}_{k=1}^{k_1}\) of \(\mathcal{U}_1(\omega^d_p)\) with minimal \(k_1\). If \(V\in\mathbf B^\mathbf F_{2\varepsilon_k}(S_k)\) for some \(k\le k_1\), then
\[
    \mathbf F(V,\mathcal{S}(\omega^d_p)) \ge \mathbf F(S_k,\mathcal{S}(\omega^d_p)) - \mathbf F(S_k,V)
    \ge 1 - 2\varepsilon_k \ge 1/2.
\]
Hence
\[
    \bigcup_{k=1}^{k_1}\mathbf B^\mathbf F_{2\varepsilon_k}(S_k) \subset \mathcal{U}_2(\omega^d_p).
\]

Inductively, suppose we have chosen \(\{S_k\}_{k=1}^{k_{m-1}}\) covering \(\mathcal{U}_{m-1}(\omega^d_p)\). Set
\[
    \mathcal R_m := \mathcal{U}_m(\omega^d_p)\setminus\bigcup_{k=1}^{k_{m-1}}\mathbf B^\mathbf F_{\varepsilon_k}(S_k),
\]
which is compact. For each \(S\in\mathcal R_m \), choose \(\varepsilon_S<\tfrac{1}{2m(m+1)}\), and then take a finite subcovering \(\{\mathbf B^\mathbf F_{\varepsilon_k}(S_k)\}_{k=k_{m-1}+1}^{k_m}\) with minimal \(k_m\). Then for any \(k\le k_m\) and \(V\in\mathbf B^\mathbf F_{2\varepsilon_k}(S_k)\) we have
\[
    \mathbf{F}\bigl(V, \mathcal{S}(\omega^d_p)\bigr) \geq \mathbf{F}\bigl(S_k, \mathcal{S}(\omega^d_p)\bigr) - \mathbf{F}(S_k, V) \geq \frac{1}{m} - 2\varepsilon_k \geq \frac{1}{m} - \frac{1}{m(m+1)} = \frac{1}{m+1}.
\]
Moreover, if \(k\ge k_{m-1}+1\), then
\begin{align*}
    \mathbf{F}\bigl(V, \mathcal{S}(\omega^d_p)\bigr) &\leq \mathbf{F}\bigl(S_k, \mathcal{S}(\omega^d_p)\bigr) + \mathbf{F}(S_k, V) \\
    &\leq \frac{1}{m-1} + 2\varepsilon_k \leq \frac{1}{m-1} + \frac{1}{m(m+1)} \\
    &< \frac{1}{m-1} + \frac{1}{m+1} = \frac{2m}{m^2 - 1} < \frac{2}{m} < \frac{1}{\lfloor m/2\rfloor}.
\end{align*}
Thus
\begin{equation}\label{eq:subset1}
    \bigcup_{k=1}^{k_m}\mathbf B^\mathbf F_{2\varepsilon_k}(S_k) \subset \mathcal{U}_{m+1}(\omega^d_p),
\end{equation}
and
\begin{equation}\label{eq:subset2}
    \bigcup_{k=k_{m-1}+1}^{k_m}\mathbf B^\mathbf F_{2\varepsilon_k}(S_k)
    \subset \mathcal{U}_{m+1}(\omega^d_p)\setminus\mathcal{U}_{\lfloor m/2\rfloor}(\omega^d_p).
\end{equation}

Continuing this process for all \(m\) yields the desired sequence \(\{S_k\}_{k\in\mathbb N}\). The minimality in each stage guarantees the first bullet. The covering property is built into the construction, giving the second bullet. By~\eqref{eq:subset1} and~\eqref{eq:subset2}, for a fixed \(k\le k_m\), any ball \(\mathbf B^\mathbf F_{2\varepsilon_{k'}}(S_{k'})\) that intersects \(\mathbf B^\mathbf F_{2\varepsilon_k}(S_k)\) must itself be among the finitely many balls produced up to stage \(2m+2\). Hence
\[
    \#\bar K(S_k)\le k_{2m+2}<\infty,
\]
and the third bullet follows.

\textbf{Step 2.} We construct a function \(\eta:\mathcal{U}(\omega^d_p)\to(0,1]\) such that:
\begin{itemize}
    \item The neighborhoods satisfy
    \[
    \begin{aligned}
    \mathcal N_\eta &:= \bigcup_{S\in\mathcal U(\omega^d_p)} \mathbf B^\mathbf{F}_{2\eta(S)}(S)
    \subset \bigcup_{k} \overline{\mathbf B}^\mathbf{F}_{\varepsilon_k}(S_k),\\
    \mathcal N^m_\eta &:= \bigcup_{S\in\mathcal U_m(\omega^d_p)} \mathbf B^\mathbf{F}_{2\eta(S)}(S)
    \subset \bigcup_{k=1}^{k_m} \overline{ \mathbf B}^\mathbf{F}_{\varepsilon_k}(S_k).
    \end{aligned}
    \]
    \item On each \(\mathcal{U}_m(\omega^d_p)\), the function \(\eta\) has a positive lower bound \(\tilde\eta_m>0\).
    \item For any \(\lambda\in\{0,1,\dots,p(n-d)\}\), \(l\in\{0,1,\dots,p(n-d)\}\), \(K\subset\mathbb{N}\), \(\delta\in\bigl(0,c(\lambda,K)/4\bigr)\), and any continuous map \(\Theta:X^l\to \mathcal{N}_\eta\cap \mathbf{B}^\mathbf{F}_{\lambda,K}\) (assume \(\mathcal{N}_\eta\cap \mathbf{B}^\mathbf{F}_{\lambda,K}\neq\emptyset\)), there exists a homotopy
    \[
        H^{\Theta,\delta}_{\lambda,K}:X^l\times[0,1]\to \mathbf{B}^\mathbf{F}_{\lambda+1,K}
    \]
    such that:
    \begin{enumerate}
        \item \(H^{\Theta,\delta}_{\lambda,K}(\cdot,0)=\Theta(\cdot)\).
        \item \(\|H^{\Theta,\delta}_{\lambda,K}(x,t)\|(M)-\|\Theta(x)\|(M)\le \delta\) for all \(t\in[0,1]\).
        \item For \(t\in[0,1/2]\), \(\mathbf{F}\bigl(H^{\Theta,\delta}_{\lambda,K}(x,t),\Theta(x)\bigr)\le \delta\).
        \item For \(t\in[1/2,1]\) there exists \(v(x)\in\overline{B}^{p(n-d)+1}\) such that
        \begin{equation}\label{eq:flow}
            H^{\Theta,\delta}_{\lambda,K}(x,t)=\phi^{\Theta(x)}_{\lambda,K}\bigl(v(x),(2t-1)T\bigr).
        \end{equation}
        \item \(H^{\Theta,\delta}_{\lambda,K}(\cdot,1)\notin \mathcal{N}_\eta\).
    \end{enumerate}
\end{itemize}
        
\textbf{Claim 1.} There exists a \emph{continuous} positive function \(\eta_1:\mathcal{U}(\omega^d_p)\to\mathbb{R}^+\) such that, for any \(m\in\mathbb{N}\) and any \(S\in\mathcal{U}_m(\omega^d_p)\),
\[
    \mathbf{B}^\mathbf{F}_{2\eta_1(S)}(S)\subset \bigcup_{k=1}^{k_m}\mathbf{B}^\mathbf{F}_{\varepsilon_k}(S_k).
\]

\begin{proof}
Define \(\eta_1\) inductively. If \(S\in\mathcal{U}_1(\omega^d_p)\), set
\[
    \eta_1(S):=\max\Bigl\{r>0:\ \mathbf{B}^\mathbf{F}_{2r}(S)\subset \bigcup_{k=1}^{k_1}\overline{\mathbf{B}}^\mathbf{F}_{\varepsilon_k}(S_k)\Bigr\}.
\]
If \(m\ge2\) is the least integer with \(S\in \mathcal{U}_m(\omega^d_p)\setminus \mathcal{U}_{m-1}(\omega^d_p)\), define
\[
    \eta_1(S):=\min\left\{
    \begin{array}{l}
    \displaystyle \max\Bigl\{r>0:\ \mathbf{B}^\mathbf{F}_{2r}(S)\subset \bigcup_{k=1}^{k_m}\overline{\mathbf{B}}^\mathbf{F}_{\varepsilon_k}(S_k)\Bigr\},\\[6pt]
    \displaystyle \min\bigl\{\eta_1(S')+\mathbf{F}(S,S'):\ S'\in \mathcal{U}_{m-1}(\omega^d_p)\bigr\}
    \end{array}\right\}>0.
\]
This \(\eta_1\) is continuous and therefore has a positive lower bound on each compact set \(\mathcal{U}_m(\omega^d_p)\).
\end{proof}

Since $\eta_1$ is continuous, the set $\{t\geq0: \forall\, V\in \overline{\mathbf{B}}^\mathbf{F}_{t}(S),\ \forall\,\lambda\in\{0,1,\dots,p(n-d)\},\bigl|\omega^d_p-\|V\|(M)\bigr|\le \tilde c(\lambda,\tilde K)/3\}$ is closed, and since $\|V\|(M)=\omega^d_p$, by continuity of mass we have the supremum of this set is positive. Therefore, we can define \(\eta(S)\) to be the largest number in \(\bigl(0,\min(\eta_1(S),1)]\bigr)\) such that
\[
    \forall\, V\in \mathbf{B}^\mathbf{F}_{\eta(S)}(S),\ \forall\,\lambda\in\{0,1,\dots,p(n-d)\}:\quad
    \bigl|\omega^d_p-\|V\|(M)\bigr|\le \tilde c(\lambda,\tilde K)/3
\]
for every nonempty subset \(\tilde K\subset \{\,k\mid \mathbf{B}^\mathbf{F}_{\eta_1(S)}(S)\cap \mathbf{B}^\mathbf{F}_{2\varepsilon_k}(S_k)\neq\emptyset\,\}\).

To show that \(\eta\) has a positive lower bound \(\tilde\eta_m\) on each \(\mathcal{U}_m(\omega^d_p)\), observe that for any \(S\in\mathcal{U}_m(\omega^d_p)\) the index set
\[
    \{\,k : \mathbf{B}^\mathbf{F}_{\eta_1(S)}(S)\cap \mathbf{B}^\mathbf{F}_{2\varepsilon_k}(S_k)\neq\emptyset\,\}
    \subset K_m:=\bigcup_{k=1}^{k_m}\bar K(S_k),
\]
since
\[\displaystyle \bigcup_{S\in\mathcal{U}_m(\omega^d_p)}\mathbf{B}^\mathbf{F}_{2\eta_1(S)}(S)\subset \bigcup_{k=1}^{k_m}\overline{\mathbf{B}}^\mathbf{F}_{\varepsilon_k}(S_k).\]
By the third bullet in \textbf{Step 1}, the set \(K_m\) is finite. It follows from the compactness of \(\mathcal{U}_m(\omega^d_p)\) that there exists \(\tilde\eta_m>0\) such that, for any \(S\in\mathcal{U}_m(\omega^d_p)\) and any \(V\in \mathbf{B}^\mathbf{F}_{\tilde\eta_m}(S)\),
\[
    \bigl|\omega^d_p-\|V\|(M)\bigr|\le \min_{\substack{K\subset K_m\\ \lambda\in\{0,1,\dots,p(n-d)\}}}\frac{\tilde c(\lambda,K)}{3}.
\]
Hence \(\eta(S)\ge \tilde\eta_m>0\) on \(\mathcal{U}_m(\omega^d_p)\).

For the third bullet, apply Lemma~\ref{lem:homotopy} to construct the homotopy \(H^{\Theta,\delta}_{\lambda,K}\). Properties (1)–(4) follow directly. To verify (5), suppose, toward a contradiction, that there exist \(x\in X^l\) and \(S\in\mathcal{U}(\omega^d_p)\) with
\[
H^{\Theta,\delta}_{\lambda,K}(x,1)\in \mathbf{B}^\mathbf{F}_{\eta(S)}(S)\cap \mathbf{B}^\mathbf{F}_{\lambda+1,K}.
\]
By the definition of \(\eta\),
\[
\omega^d_p-\|H^{\Theta,\delta}_{\lambda,K}(x,1)\|(M)<\tilde c(\lambda,K)/2.
\]
Since \(H^{\Theta,\delta}_{\lambda,K}(x,0)=\Theta(x)\in \mathbf{B}^\mathbf{F}_{\eta(S)}(S)\cap \mathbf{B}^\mathbf{F}_{\lambda+1,K}\), we also have
\[
\|\Theta(x)\|(M)-\omega^d_p<\tilde c(\lambda,K)/2.
\]
Therefore,
\[
\|\Theta(x)\|(M)-\|H^{\Theta,\delta}_{\lambda,K}(x,1)\|(M)<\tilde c(\lambda,K),
\]
contradicting item (2) of Lemma~\ref{lem:homotopy}. This proves (5).

\textbf{Step 3.} We show that the homotopy map defined in \textbf{Step 2} does not push a varifold too close to $\mathcal{U}(\omega^d_p)$, provided it is initially far from $\mathcal{U}(\omega^d_p)$.

More precisely, we claim that for any $S\in\mathcal{U}(\omega^d_p)$ there exist sequences $\{\varepsilon_q(S)\}_{q=1}^{p(n-d)}\subset\mathbb{R}^+$ and $\{a_q(S)\}_{q=1}^{p(n-d)+1}$ with $a_1(S)=2$ and $a_q(S)\ge 2^q$ such that the following holds. In the third bullet of \textbf{Step 2}, for any $q\in\{1,\dots,p(n-d)\}$, any $S\in\mathcal{U}(\omega^d_p)$ with $\mathbf{B}^{\mathbf F}_{\lambda+1,K}\cap \mathbf{B}^{\mathbf F}_{\eta(S)}(S)\neq\emptyset$, and any $x\in X^l$ with $\Theta(x)\notin \mathbf{B}^{\mathbf F}_{\eta(S)/a_q(S)}(S)$ and $\|\Theta(x)\|(M)-\omega^d_p\le \varepsilon_q(S)$, we have for all $t\ge0$ and all $\delta\in\bigl(0,\min\bigl(\eta(S)/(4a_q(S)),\varepsilon_q(S)\bigr)\bigr)$,
\[
    H^{\Theta,\delta}_{\lambda,K}(x,t)\notin \mathbf{B}^{\mathbf F}_{\eta(S)/a_{q+1}(S)}(S).
\]
Moreover, on each $\mathcal{U}_m(\omega^d_p)$ the quantities $\varepsilon_q(S)$ admit a positive lower bound $\tilde\varepsilon^m_q$, and $a_q(S)$ admit a uniform upper bound $\tilde a^m_q$.

We construct $a_q(S)$ and $\varepsilon_q(S)$ inductively. Set $a_1(S)=2$.

Fix $S\in\mathcal{U}(\omega^d_p)$ and suppose $a_q(S)$ is defined (and if $q>1$, also $\varepsilon_{q-1}(S)$). For each \(k\) with \(\overline{\mathbf{B}}^{\mathbf F}_{\lambda+1,\{k\}}\cap \mathbf{B}^{\mathbf F}_{\eta(S)}(S)\neq\emptyset\) and for each \(V\in\overline{\mathbf{B}}^{\mathbf F}_{\lambda+1,\{k\}}\), set
\[
\begin{aligned}
A_{q,S,k}(V)&:=\bigl((F^k_{\cdot})_{\#}V\bigr)^{-1}\partial\mathbf{B}^{\mathbf F}_{\,3\eta(S)/(4a_q(S))}(S),\\
B_{q,S,k}(V)&:=\bigl((F^k_{\cdot})_{\#}V\bigr)^{-1}\overline{\mathbf{B}}^{\mathbf F}_{\,\eta(S)/(2a_q(S))}(S).
\end{aligned}
\]
Define
\[
d_{q,S,k}(V):=\mathrm{dist}\bigl(A_{q,S,k}(V),\,B_{q,S,k}(V)\bigr),
\]
with the convention that \(\mathrm{dist}(\emptyset,\cdot)=\mathrm{dist}(\cdot,\emptyset)=+\infty\). By Lemma~\ref{lem:dist_F},
\[
    d_{q,S,k}(V)\ \ge\ h_{S_k}\!\left(\frac{\eta(S)}{4a_q(S)}\right).
\]
Choose $m \in \mathbb{N}$ such that $S \in \mathcal{U}_m(\omega^d_p)$. The index set
\[
    \{\,k : \overline{\mathbf{B}}^\mathbf{F}_{\lambda+1,\{k\}}\cap \mathbf{B}^\mathbf{F}_{\eta(S)}(S)\neq\emptyset\,\}
    \subset K_m=\bigcup_{k=1}^{k_m}\bar K(S_k),
\]
since
\[\displaystyle \bigcup_{S\in\mathcal{U}_m(\omega^d_p)}\mathbf{B}^\mathbf{F}_{2\eta_1(S)}(S)\subset \bigcup_{k=1}^{k_m}\mathbf{B}^\mathbf{F}_{\varepsilon_k}(S_k).\]
By the third bullet in \textbf{Step 1}, the set \(K_m\) is finite. Hence we can set
\begin{equation}\label{eq:def-eps}
    \varepsilon_q(S):=\min_{\substack{k:\\ \overline{\mathbf{B}}^{\mathbf F}_{\lambda+1,\{k\}}\cap \mathbf{B}^{\mathbf F}_{\eta(S)}(S)\neq\emptyset}}
    \frac{c_{0,k}}{16}\left[h_{S_k}\!\left(\frac{\eta(S)}{4a_q(S)}\right)\right]^2.
\end{equation}

\textbf{Claim 2.} There exists $\tilde\varepsilon^m_q>0$ such that $\inf_{S\in\mathcal{U}_m(\omega^d_p)}\varepsilon_q(S)\ge \tilde\varepsilon^m_q$.
\begin{proof}
For fixed $m$, by the first bullet of \textbf{Step 2},
\[
K'_m:=\Bigl\{k:\ \exists\,S\in\mathcal{U}_m(\omega^d_p)\ \text{with}\ \overline{\mathbf{B}}^{\mathbf F}_{\lambda+1,\{k\}}\cap \mathbf{B}^{\mathbf F}_{\eta(S)}(S)\neq\emptyset\Bigr\}
\subset \bigcup_{k=1}^{k_m}\bar K(S_k),
\]
hence $\#K'_m<\infty$ by the third bullet of \textbf{Step 1}. By the inductive hypothesis and the second bullet of \textbf{Step 2}, the ratios $\eta(S)/a_q(S)$ admit a positive lower bound $c'_m>0$ on $\mathcal{U}_m(\omega^d_p)$. Therefore,
\[
\begin{aligned}
\inf_{S\in\mathcal{U}_m(\omega^d_p)}\varepsilon_q(S)
&= \min_{k\in K'_m}\ 
\inf_{\substack{S\in\mathcal{U}_m(\omega^d_p)\\
\overline{\mathbf{B}}^{\mathbf F}_{\lambda+1,\{k\}}
\cap \mathbf{B}^{\mathbf F}_{\eta(S)}(S)\neq\emptyset}}
\frac{c_{0,k}}{16}\left[h_{S_k}\!\left(\frac{\eta(S)}{4a_q(S)}\right)\right]^2 \\
&\ge \min_{k\in K'_m}\frac{c_{0,k}}{16}\,h_{S_k}(c'_m)^2 \;>\; 0.
\end{aligned}
\]
\end{proof}

Choose $a_{q+1}(S)\ge 2a_q(S)+1$ minimal such that for every $V\in\overline{\mathbf{B}}^{\mathbf F}_{\eta(S)/a_{q+1}(S)}(S)$,
\[
\|V\|(M)\ \ge\ \omega^d_p-\varepsilon_q(S).
\]

\textbf{Claim 3.} There exists $\tilde a^m_{q+1}>0$ such that $\sup_{S\in\mathcal{U}_m(\omega^d_p)} a_{q+1}(S)\le \tilde a^m_{q+1}$.
\begin{proof}
Since $\eta\le 1$, it suffices to find $c''_m>0$ such that for every $S\in\mathcal{U}_m(\omega^d_p)$ and every $V\in\overline{\mathbf{B}}^{\mathbf F}_{c''_m}(S)$,
\[
\|V\|(M)\ \ge\ \omega^d_p-\tilde\varepsilon^m_q.
\]
Then $c''_m\le \eta(S)/a_{q+1}(S)\le 1/a_{q+1}(S)$, so $a_{q+1}(S)\le 1/c''_m$ uniformly on $\mathcal{U}_m(\omega^d_p)$.

Argue by contradiction: suppose there exist $S_i\in\mathcal{U}_m(\omega^d_p)$ and $V_i\in\overline{\mathbf{B}}^{\mathbf F}_{1/i}(S_i)$ with $\|V_i\|(M)<\omega^d_p-\tilde\varepsilon^m_q$. By compactness of $\mathcal{U}_m(\omega^d_p)$, up to a subsequence $S_i\to S$ and $V_i\to S$, so $\|S\|(M)=\omega^d_p$, a contradiction.
\end{proof}

\textbf{Claim 4.} If $\Theta(x)\notin \mathbf{B}^{\mathbf F}_{\eta(S)/a_q(S)}(S)$ and $\|\Theta(x)\|(M)-\omega^d_p\le \varepsilon_q(S)$, then
\[
H^{\Theta,\delta}_{\lambda,K}(x,t)\notin \mathbf{B}^{\mathbf F}_{\eta(S)/a_{q+1}(S)}(S).
\]
\begin{proof}
For $t\in[0,1/2]$, since $\delta<\eta(S)/(4a_q(S))$ and by the third bullet of \textbf{Step 2},
\[
\begin{aligned}
\mathbf{F}\bigl(H^{\Theta,\delta}_{\lambda,K}(x,t),S\bigr)
&\ge \mathbf{F}\bigl(\Theta(x),S\bigr)
 - \mathbf{F}\bigl(H^{\Theta,\delta}_{\lambda,K}(x,t),\Theta(x)\bigr) \\
&\ge \frac{\eta(S)}{a_q(S)} - \delta
\ \ge\ \frac{3\,\eta(S)}{4\,a_q(S)}
\ \ge\ \frac{\eta(S)}{a_{q+1}(S)}.
\end{aligned}
\]

For $t\in(1/2,1]$, we have $\delta<\varepsilon_q(S)$, hence by the third bullet of \textbf{Step 2},
$\|H^{\Theta,\delta}_{\lambda,K}(x,1/2)\|(M)\le \omega^d_p+2\varepsilon_q(S)$. If, toward a contradiction, there exists $t_1\in(1/2,1]$ with
$H^{\Theta,\delta}_{\lambda,K}(x,t_1)\in\overline{\mathbf{B}}^{\mathbf F}_{\eta(S)/a_{q+1}(S)}(S)$, let $k=\bar k(K)$. Then
\[
\begin{aligned}
\mathrm{dist}\!\Big(
\bigl((F^k_{\cdot})_{\#}(\Theta(x))\bigr)^{-1}\!\bigl(&H^{\Theta,\delta}_{\lambda,K}(x,t_1)\bigr),\ 
\bigl((F^k_{\cdot})_{\#}(\Theta(x))\bigr)^{-1}\!\bigl(H^{\Theta,\delta}_{\lambda,K}(x,1/2)\bigr)
\Big) \\
&\ge d_{q,S,k}(\Theta(x))
\ \ge\ h_{S_k}\!\left(\frac{\eta(S)}{4a_q(S)}\right) \;>\; 0.
\end{aligned}
\]

By Lemma~\ref{lem:grad}, in addition to equations \eqref{eq:flow} and \eqref{eq:def-eps} above,
\[
\|H^{\Theta,\delta}_{\lambda,K}(x,t_1)\|(M)
\ \le\ \|H^{\Theta,\delta}_{\lambda,K}(x,1/2)\|(M)-\frac{c_{0,k}}{2}\left[h_{S_k}\!\left(\frac{\eta(S)}{4a_q(S)}\right)\right]^2
\ \le\ \omega^d_p-6\varepsilon_q(S),
\]
contradicting the defining property of $a_{q+1}(S)$.
\end{proof}

Finally, define
\[
\varepsilon_m:=\min_{q\in\{1,\dots,p(n-d)\}}\ \inf_{S\in\mathcal{U}_m(\omega^d_p)}
\min\!\left(\frac{\eta(S)}{4a_{q+1}(S)},\,\varepsilon_q(S)\right)>0.
\]

\textbf{Step 4 [Hierarchical Deformations].}
In this step we construct a (sub)sequence of homotopies
\[
H_i: X_i \times [0,1] \longrightarrow \mathcal Z_d(M; \mathbb Z_2)
\]
such that \(H_i(\cdot,0)=\Phi_i\) and \(H_i(\cdot,1)=\Psi_i\), with all required properties.

Choose a subsequence \(\{j_i\}\) so that \(\sup_x \mathbf M(\Phi_{j_i}(x)) \le \omega^d_p + \varepsilon_i/2\), which is possible by the definition of a min--max sequence. Relabel this subsequence as \(\Phi_i\).

It suffices to deform \(\Phi_i\) to \(\Psi_i\) so that
\(|\Psi_i|\cap \mathcal N^i_{\eta/(2a_{p(n-d)+1})}=\emptyset\), where for
\(c:\mathcal U(\omega^d_p)\to\mathbb R^+\) we set
\[
\mathcal N^i_{c\eta}\ :=\ \bigcup_{S\in\mathcal U_i(\omega^d_p)} \mathbf B^{\mathbf F}_{\,2c(S)\eta(S)}(S).
\]
Fix \(i\) for the remainder of this step.

Recall that subdivision of a simplicial complex does not change the topological or metric properties, only the combinatorial structure. We perform repeated barycentric subdivisions of \(X_i\) to obtain a “nice’’ combinatorial structure and partition the closed faces into \textit{good} and \textit{bad} faces, with the following properties:
\begin{enumerate}
  \item If \(|\Phi_i|(F)\cap \mathcal N^i_{\eta/2}\neq\emptyset\), then \(|\Phi_i|(F)\subset \mathcal N^i_\eta\). Denote by \(\mathcal B\) the set of all such \textit{bad} faces.
  \item For each \(F\in\mathcal B\) there is a nonempty finite set \(K(F)\subset\mathbb N\) such that \(|\Phi_i|(F)\subset \mathbf B^{\mathbf F}_{K(F)}\), and whenever \(F\subset F'\) with \(F,F'\in\mathcal B\) we have \(K(F)\supset K(F')\).
\end{enumerate}
Since \(X_i\) is compact, \(|\Phi_i|\) is uniformly continuous. Moreover, \(\eta\) has a uniform positive lower bound on \(\mathcal U_i(\omega^d_p)\); hence \(\mathbf F(\partial\mathcal N^i_\eta,\mathcal N^i_{\eta/2})\ge c_i>0\). After sufficiently many barycentric subdivisions, property (1) holds.

By the first bullet of \textbf{Step 2}, each $|\Phi_i|(F)$ with $F\in\mathcal B$ can be covered by finitely many
\begin{equation*}
    \{\mathbf{B}^\mathbf{F}_{\varepsilon_j}(S_j)\}.
\end{equation*}
Hence the pullback open cover
\begin{equation*}
    \{\,|\Phi_i|^{-1}(\mathbf{B}^\mathbf{F}_{\varepsilon_j}(S_j))\,\}
\end{equation*}
is finite and therefore admits a Lebesgue number. Because $\mathcal B$ is finite, after further barycentric subdivisions each $|\Phi_i|(F)$ with $F\in\mathcal B$ lies in one of the $\mathbf{B}^\mathbf{F}_{\varepsilon_j}(S_j)$.

\begin{figure}
\centering
\includegraphics[scale=1.2]{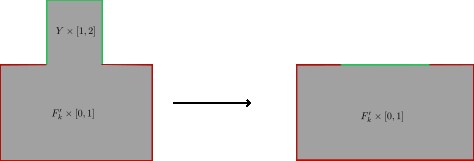}
\caption{Obtaining a homotopy on $F'_k$ from one on $Y$}
\label{YFkk}
\end{figure}

We define the index sets $K$ as follows. If $F\in\mathcal B$ is not contained in any other member of $\mathcal B$, set $K(F)$ to be the singleton consisting of one index $j$ such that $\mathbf{B}^\mathbf{F}_{\varepsilon_j}(S_j)$ contains $|\Phi_i|(F)$. Proceeding by downward induction on dimension, for each face $F'\subset F$ with $F\in\mathcal B$, define
\begin{equation*}
K(F') := \bigcup_{F'\subset F\in\mathcal B} K(F).
\end{equation*}
By construction, this has the required properties. We cannot simply set $K(F')=K(F)$ because a lower-dimensional face $F'$ may lie in multiple higher-dimensional faces $F$ in $\mathcal B$, which can yield different sets.

We now define \(H_i\) on skeleta \(X_i^{(k)}\) by induction on \(k\). For $k = 0$, we apply \textbf{Step 2} with $\lambda = 0$ and $\delta < \varepsilon_i/4$ to all the $0$-cells in $\mathcal{B}$. For all the $0$-cells outside $\mathcal{B}$, we simply use the constant homotopy map. This yields $H_i^{(0)}: X_i^{(0)}\times[0,1]\to \mathcal Z_d(M;\mathbb Z_2)$ such that
\begin{equation*}
    |H_i^{(0)}|(X_i^{(0)}\times\{1\})\ \cap\ \mathcal N^i_{\eta/a_1}\ =\ \emptyset,
    \qquad
    |H_i^{(0)}|(x,[0,1])\ \subset\ \mathbf B^{\mathbf F}_{\,1,K(x)}
    \ \text{ for }x\in\mathcal B^{(0)}.
\end{equation*}

We assume the homotopy map has been defined on the $(k-1)$-skeleton, with $k\ge 1$. Consider a $k$-face $F_k$. Set
\begin{equation*}
F'_k := F_k \cup \bigl(\partial F_k \times [0,1]\bigr).
\end{equation*}
Then $F'_k$ is homeomorphic to $F_k$ ($\cong D^k$), so we can define $\Theta$ on $F'_k$ by concatenating $H^{(k-1)}_i$ on $\partial F_k \times [0,1]$ (available by the inductive hypothesis) with $\Phi_i$ on $F_k$. Now we construct a homotopy map $\tilde H^{(k)}_i$ with initial data $\Theta$. We have two cases:

\textbf{Case 1.} If $F_k \notin \mathcal{B}$, then according to the definition of $\mathcal{B}$, no cell in $\partial F_k$ belongs to $\mathcal{B}$. In this case we define the homotopy on this cell to be constant.

\textbf{Case 2.} If \(F_k\in\mathcal B\), the subdivision guarantees
\(|\Theta|(F_k')\subset \mathbf B^{\mathbf F}_{\,k,K(F_k)}\).

However, it may fail that \(|\Theta|(F_k')\subset\mathcal N^i_\eta\). Subdivide \(F_k'\) further so that every closed \(k\)-face \(\widetilde F\) with \(|\Theta|(\widetilde F)\cap \mathcal N^i_{\eta/a_k}\neq\emptyset\) actually satisfies \(|\Theta|(\widetilde F)\subset \mathcal N^i_\eta\). Let \(Y\) be the union of all such \(\widetilde F\). By induction, \(|\Theta|(\partial F_k')\cap \mathcal N^i_{\eta/a_k}=\emptyset\), hence \(|\Theta|(\partial Y)\cap \mathcal N^i_{\eta/a_k}=\emptyset\).

With this new subdivision, we can define a map 
\begin{equation*}
    \hat H: F'_k \times [0, 1] \cup Y \times [1,2] \rightarrow \mathcal{Z}_d(M; \mathbb{Z}_2) 
\end{equation*}
\begin{figure}
\centering
\includegraphics[scale=1.2]{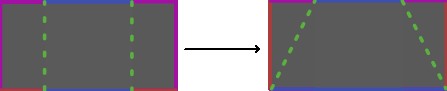}
\caption{Obtaining a homotopy on $F_k$ from one on $F'_k$}
\label{F'Fk}
\end{figure}
such that $\hat H(\cdot, t) := \Theta(\cdot)$ for $t \in [0,1]$, and $\hat H(\cdot, t+1) = H^{|\Theta|, \delta}_{k, K(F_k)}(\cdot, t)$ for $t \in [0,1]$ in \textbf{Step 2} with $\delta < \varepsilon_i$. By the construction of $Y$, we have
\begin{equation*}
    |\hat H((F'_k - Y) \times 1 \cup \partial F'_k \times [0, 1])| \cap \mathcal{N}^i_{\eta/a_{k+1}} = \emptyset\,.
\end{equation*}
By (5) in the third bullet of \textbf{Step 2},
\begin{equation*}
    |\hat H(Y \times 2)| \cap \mathcal{N}^i_{\eta/a_{k+1}} = \emptyset\,.
\end{equation*}
By \textbf{Step 3},
\begin{equation*}
    |\hat H(\partial Y \times [1, 2])| \cap \mathcal{N}^i_{\eta/a_{k+1}} = \emptyset\,.
\end{equation*} 
We can derive $\tilde H^{(k)}_i: F'_k \times [0, 1] \rightarrow \mathcal{Z}_d(M; \mathbb{Z}_2)$ from $\hat H$ induced by the homeomorphism (see Figure \ref{YFkk}.)
\begin{equation*}
    \begin{aligned}
        (F'_k \times [0,1], &\partial F'_k \times [0,1], F'_k \times 1) \cong 
        (F'_k \times [0, 1] \cup Y \times [1,2], \\
        & \partial F'_k \times [0,1], (F'_k\backslash Y)\times 1) \cup \left(Y \times 2\right) \cup \left(\partial Y \times [1,2]\right))\,.
    \end{aligned}            
\end{equation*}
        
By the property of $\hat H$, we also have
\begin{equation*}
    |\tilde H^{(k)}_i(F'_k \times 1 \cup \partial F'_k \times [0, 1])| \cap \mathcal{N}^i_{\eta / a_{k + 1}} = \emptyset\,.
\end{equation*}
For both cases, we can derive $H^{(k)}_i: F_k \times [0, 1] \rightarrow \mathcal{Z}_d(M; \mathbb{Z}_2)$ from $\tilde{H}^{(k)}_i$ induced by the homeomorphism 
\begin{equation*}
    (F_k \times [0,1], F_k \times 1) \cong (F'_k \times [0,1], F'_k \times 1 \cup \partial F'_k \times [0,1]),
\end{equation*}
satisfying that $H^{(k)}_i|_{\partial F_k \times [0,1]} = H^{(k-1)}_i|_{\partial F_k \times [0,1]}$ and $H^{(k)}_i|_{F_k}(\cdot, 0) = \Phi_i|_{F_k}$ (see Figure \ref{F'Fk}). Concatenating over all \(k\)-cells yields
\begin{equation*}
    H^{(k)}_i: X^{(k)}_i \times [0,1] \rightarrow \mathcal{Z}_d(M; \mathbb{Z}_2).
\end{equation*}
Note that
\begin{equation*}
    |H^{(k)}_i(F_k \times 1)| = |\tilde H^{(k)}_i(F'_k \times 1 \cup \partial F'_k \times [0, 1])|\,.   
\end{equation*}
Therefore, we can conclude that
\begin{equation*}
    |H^{(k)}_i|(X^{(k)}_i \times 1) \cap \mathcal{N}^i_{\eta / a_{k+1}} = \emptyset\,.
\end{equation*}

Set \(\Psi_i:=H_i^{(p(n-d))}(\cdot,1)\) and \(\bar\varepsilon(S):=\eta(S)/(2a_{p(n-d)+1}(S))\). Then all required conclusions follow.
\end{proof}

\begin{proof}[Proof of Theorem \ref{thm:intro-main}]
    For any min-max sequence $\{\Phi_i\}$ for the $d$-dimensional $p$-width, by Theorem \ref{Thm:deform}, there exists a new min-max sequence $\{\Psi_i\}$ such that $\mathbf{C}(\{\Psi_i\})\cap \mathcal{APR}_{p,d} \subset \mathcal{S}(\omega^d_p)$. Therefore, the Almgren--Pitts theory generates $S \in \mathcal{IV}_d(M)$ that is stationary and almost minimizing in annuli such that $S \in \mathbf{C}(\{\Psi_i\}) \cap \mathcal{APR}_{p,d} \subset \mathcal{S}(\omega^d_p)$, i.e., $\mathrm{index}(S) \leq p(n-d)$.
\end{proof}

\section{Optimality of the bound} \label{sec-optimal}

In this section, we prove that the index bound must depend on $n$ and that it is optimal for $\omega_1^1$.

\begin{theorem}
    For any $n \geq 2$, $\omega_1^1(S^n, g_{\operatorname{round}}) = 2\pi$ and is realized by closed equatorial geodesics of Morse index $n - 1$.
\end{theorem}

\begin{proof}
    By Theorem~1 in~\cite{Rademacher}, we have $\omega_1^1(S^n, g_{\operatorname{round}}) \leq 2\pi$. Suppose, for contradiction, that $\omega_1^1(S^n, g_{\operatorname{round}}) < 2\pi$. Then Theorem~\ref{thm:1d-minmax} implies that there exists a stationary geodesic net $G$ on $(S^n, g_{\operatorname{round}})$ with $\operatorname{length}_g(G) < 2\pi$. Proposition~3.6 in~\cite{aiex} implies that the density of $G$ is less than $1$ everywhere, which is a contradiction. Hence $\omega_1^1(S^n, g_{\operatorname{round}}) = 2\pi$.

    Let $G$ be a stationary geodesic net of length $\omega_1^1(S^n, g_{\operatorname{round}})$. Since $\omega_1^1(S^n, g_{\operatorname{round}}) = 2\pi < 4\pi$, Proposition~3.6 in~\cite{aiex} implies that the density of $G$ is less than $2$ everywhere. By Corollary~3.3 in~\cite{aiex}, only triple junctions can occur in $G$. However, triple junctions are also impossible, since the density at such junctions is not an integer, which contradicts Theorem~4.13 in~\cite{aiex}. This implies that the only realizations of $\omega_1^1(S^n, g_{\operatorname{round}})$ are equatorial closed geodesics. Their Morse index is equal to $n - 1$ (see Example~2.5.7 in~\cite{Klingenberg}).
\end{proof}

\printbibliography

\vspace{0.5cm} 
\noindent Department of Mathematics, University of Toronto, Toronto, Canada\\
\textit{E-mail address}: \texttt{mitchell.gaudet@mail.utoronto.ca}

\vspace{0.5cm} 
\noindent Department of Mathematics, University of Toronto, Toronto, Canada\\
\textit{E-mail address}: \texttt{talant.talipov@mail.utoronto.ca}

\end{document}